\documentclass[12pt]{amsart}
\usepackage[english]{babel}
\usepackage{amsmath}
\usepackage{amsthm}
\usepackage{mathtools}
\usepackage{mathrsfs}
\usepackage{amssymb}
\usepackage{xcolor}
\usepackage{xfrac}
\usepackage{esint}
\usepackage{graphicx}
\usepackage{float}
\usepackage{array}
\usepackage{enumitem}
\usepackage[new]{old-arrows}
\usepackage[all,cmtip]{xy}
\usepackage{tikz-cd}
\usepackage{centernot}
\usepackage{stmaryrd}
\usepackage{comment}
\excludecomment{confidential}

\makeatletter
\@namedef{subjclassname@2020}{\textup{2020} Mathematics Subject Classification}
\makeatother

\numberwithin{equation}{section}

\newtheorem{theorem}{Theorem}[section]

\newtheorem{lemma}[theorem]{Lemma}

\newtheorem{definition}[theorem]{Definition}

\DeclareMathOperator{\Alg}{Alg}

\DeclareMathOperator{\alg}{alg}

\newcommand{\N}{\mathbb{N}}
\newcommand{\Z}{\mathbb{Z}}

\newcommand{\R}{\mathbb{R}}

\newcommand{\TT}{\mathbb{T}}

\def\a{\alpha}
\def\b{\beta}
\def\e{\varepsilon}

\def\d{\delta}

\def\G{\Gamma}
\def\l{\lambda}

\def\L{\Lambda}

\def\O{\Omega}

\def\r{\rho}

\def\s{\sigma}

\def\v{\varphi}

\DeclareMathOperator{\spann}{span}
\newcommand{\mc}{\mathcal}
\newcommand{\mf}{\mathfrak}

\begin{document}

\title[Nonlocal Kuramoto--Sivashinsky equation]{A bifurcation theory approach to the nonlocal Kuramoto--Sivashinsky equation}

\author[P. Cubillos]{Pablo Cubillos} \thanks{}
\address{Departamento de Análisis Matemático y Matemática Aplicada \\
	Universidad Complutense de Madrid (UCM) \\
	Madrid, 28040, Spain.}
\email{pacubill@ucm.es}

\author[R. Granero-Belinchón]{Rafael Granero-Belinchón} \thanks{R. G-B. was supported by the project ``An\'alisis Matem\'atico Aplicado
y Ecuaciones Diferenciales" Grant PID2022-141187NB-I00 funded by MCIN/AEI/10.13039/501100011033/FEDER, UE}
\address{Departamento de Matemáticas, Estadística y Computación \\
	Universidad de Cantabria (UC) \\
	Santander, 39005, Spain.}
\email{rafael.granero@unican.es}

\author[J. C. Sampedro]{Juan Carlos Sampedro} \thanks{J. C. S. and P. C. have been supported by the Ministry of Science and Innovation of Spain under the
	Research Grant PID2024–155890NB-I00 and by the Institute of Interdisciplinary Mathematics of the Complutense University of Madrid}
\address{Departamento de Matemática Aplicada y Ciencias de la Computación \\
	Avenida de los Castros 46 \\
	Universidad de Cantabria (UC) \\
	Santander, 39005, Spain.}
\email{juancarlos.sampedro@unican.es}

\begin{abstract}
	We study the nonlocal Kuramoto--Sivashinsky equation on the one-dimensional torus, 
	\[
	u_t+u u_x=\L^{r}u-\varepsilon \L^{s}u,\qquad x\in\mathbb T,
	\]
	where $\varepsilon>0$, $s>1$, $r\in[-1,s)$. We first prove local and global well-posedness for initial data in $H^{3}(\mathbb T)$. We then investigate the steady-state problem and show that the trivial branch undergoes bifurcation at the critical values $\varepsilon_k=k^{\,r-s}$, $k\in\mathbb N$. Using the Crandall--Rabinowitz theorem we obtain smooth local curves of nontrivial equilibria emanating from each $(\varepsilon_k,0)$ and compute the bifurcation direction. To address the global continuation of these branches we derive global \emph{a priori} bounds and apply a global alternative based on the Fitzpatrick--Pejsachowicz--Rabier degree for Fredholm maps of index zero. In particular, for the component bifurcating from the first critical point we prove that its $\varepsilon$-projection contains the interval $(2^{r-s},1)$, yielding the existence of nontrivial steady states for that parameter range. We complement the theory with numerical continuation results illustrating the bifurcation diagram and solution profiles.
\end{abstract}


\maketitle

\section{Introduction}
Although there are different ideas on whether there is a unified subject called ``\emph{Partial Differential Equations}" \cite{brezis1998partial,klainerman2010pde}, there are different questions that are ubiquitous in the community of people studying PDEs. Of course, the first questions one may want to answer are related to the definition of \emph{well-posed problem in the sense of Hadamard.} Namely, given a PDE, one may want to ensure the existence of a maybe local in time solution (in a certain suitable sense), the uniqueness of such solution (at least in certain class of functions) and the continuous dependence of the solution on the initial and, if necessary, boundary data. The answers to these three questions are basic in order applied scientists can use certain PDE as a predictive tool. Once these very initial problems are solved, two natural questions arises: Given the solution of a PDE, does such a solution exists globally in time? and, if so, what is the behavior of such solution? It is in this latter one that we focus our efforts.

In this paper we consider the following nonlinear and nonlocal partial differential equation
\begin{equation}
	\label{Eq1.1}
u_t+	uu_{x}= \L^{r}u-\varepsilon \L^{s}u, \quad x\in \mathbb{T},
\end{equation}
$$
u(x,0)=u_0(x),
$$
where $u_0$ is an odd periodic initial data, $\varepsilon>0$, $r\in[-1,s)$ and $s > 1$ are fixed parameters, $\mathbb{T}=[-\pi,\pi]$ with periodic boundary conditions and the spatial  operators are given as the following Fourier multipliers
$$
\L^{\alpha} u : = \sum_{n\in\Z,n\neq0} |n|^{\alpha} \widehat{u}(n) \frac{e^{inx}}{\sqrt{2\pi}}, \quad x\in \mathbb{T}, \quad \widehat{u}(n)=\frac{1}{\sqrt{2\pi}}\int_{\TT}u(x)e^{inx}\; dx.
$$
The previous equation \eqref{Eq1.1} appears in different problems in mathematical physics. On the one side, when $r=0$ and $s=4$ the equation goes by the name of Burgers--Sivashinsky equation \cite{goodman1994stability}, when $r=2$ and $s=4$, the equation is the celebrated Kuramoto--Sivashinsky equation \cite{nicolaenko1985some,tadmor1986well} and when $0\leq r<s$ and $1<s\leq2$ it is usually referred as the nonlocal Kuramoto--Sivashinsky equation \cite{granero2015nonlocal}. 

We will see below (see also \cite{granero2015nonlocal,nicolaenko1985some,tadmor1986well}) that the previous equation \eqref{Eq1.1} has global in time classical solutions regardless of the size of the initial data (assuming that the initial data has the required regularity). Furthermore, it has been reported that the PDE \eqref{Eq1.1} exhibits spatio-temporal chaos \cite{nicolaenko1985some} when $\varepsilon$ is small enough and that $U(x)\equiv0$ is an stable steady state if $\varepsilon$ is large enough. Our goal in this paper is to complete the bifurcation diagram for the previous equation. In other words, we want to perform a bifurcation analysis of the possible steady states of equation \eqref{Eq1.1}. Hence, we study periodic odd solutions of the nonlinear and nonlocal equation
\begin{equation}
	\label{Eq1.2}
	uu_{x}= \L^{r}u-\varepsilon \L^{s}u, \quad x\in \mathbb{T}.
\end{equation}
In particular, the main results of this paper are aimed to rigorously prove the following bifurcation diagram.

\begin{center}
	\begin{figure}[h!]
		\label{F1}
		\includegraphics[scale=0.75]{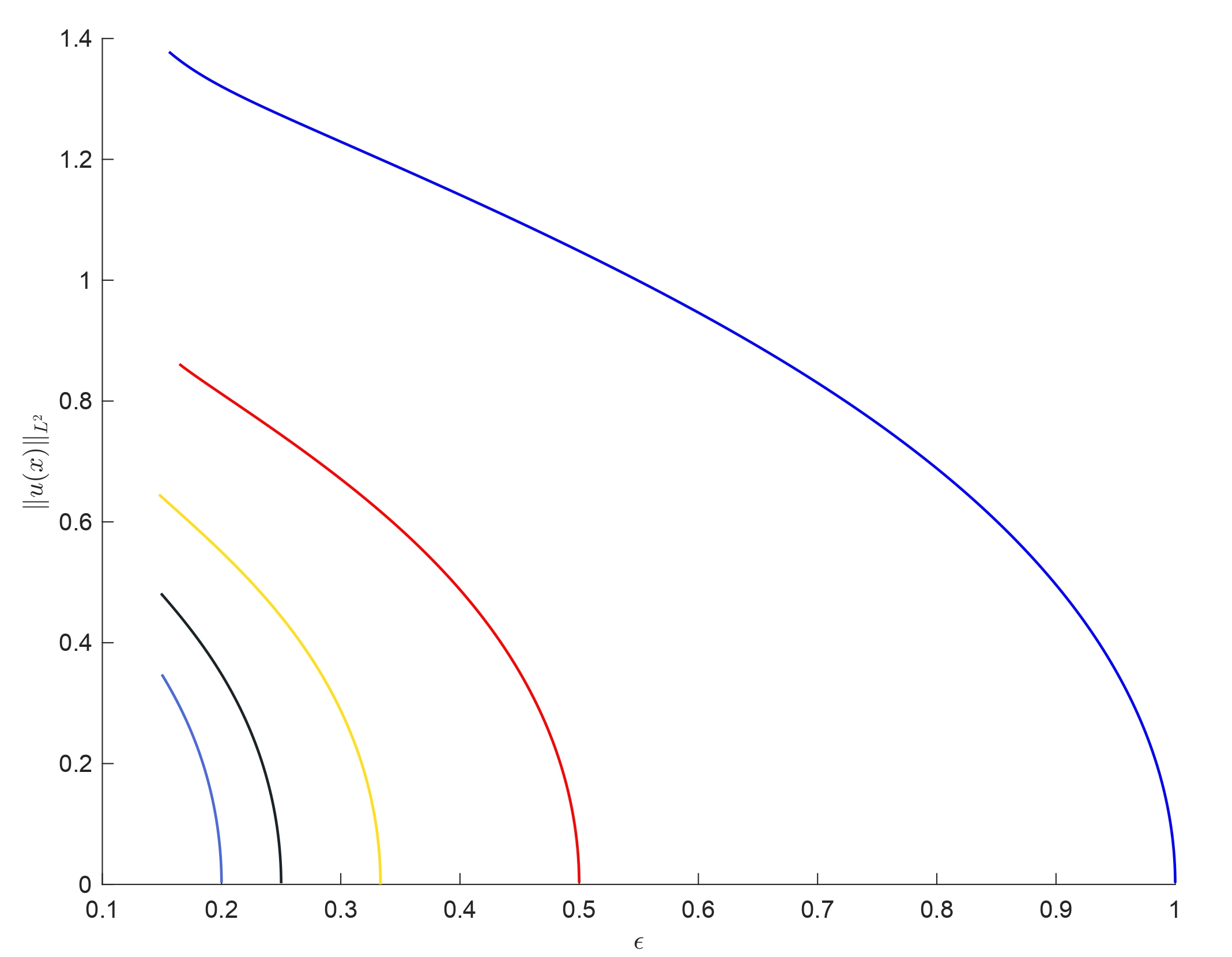}
		\caption{Bifurcation diagram for $r=\tfrac{1}{2}$ and $s=\tfrac{3}{2}$.}
	\end{figure}
\end{center}

In fact, we establish that there is a discrete sequence of values of $\varepsilon$ such that, using local bifurcation analysis (see Theorem \ref{Th3.6}), there exists a unique non-trivial steady state solution to \eqref{Eq1.1} locally near these points. More precisely, these are given by $\e_{k}=k^{r-s}$, $k\in\N$. For instance, in the particular case of Figure \ref{F1}, that is, $r=\tfrac{1}{2}$ and $s=\tfrac{3}{2}$, these are $\e_{k}=k^{-1}$, $k\in\N$. Furthermore, using global bifurcation analysis (see Theorem \ref{Th5.4}), we are able to extend the range of possible $\varepsilon$ for the first bifurcation branch. In fact, we are able to prove the existence of a non-trivial steady state to \eqref{Eq1.1} for the range
$$
\varepsilon\in(2^{r-s},1).
$$
Finally, we establish a qualitative alternative describing the possible behaviour of nontrivial steady states as $\varepsilon\to0^+$ (see Theorem~\ref{Th5.5}). More precisely, given any sequence of solutions $(\varepsilon_n,u_n)$ with $\varepsilon_n\to0$, either the sequence develops large high-frequency content in the sense that
\[
\limsup_{n\to+\infty}\|u_n\|_{\dot H^{s/2}}=+\infty,
\]
or the solutions become small in amplitude, namely $\|u_n\|_{L^\infty}\to0$ (and the two scenarios are not mutually exclusive).

We conclude the introduction by briefly describing the abstract bifurcation tools that underpin our main results.
The global alternative (see Theorem \ref{Th5.4}) we use was developed by L\'opez-G\'omez and Mora-Corral in \cite{LGMC} and later refined in \cite{LGSM} using the Fitzpatrick--Pejsachowicz--Rabier degree \cite{FPRa,FPRb,PR}. This degree extends the classical Leray--Schauder degree to nonlinear Fredholm maps of index zero, and it is therefore well suited to settings in which the relevant operators are not compact perturbations of the identity.

To the best of our knowledge, the use of Fredholm-degree methods in the study of nonlocal periodic problems of the form \eqref{Eq1.2} is still rather limited (see also \cite{SP}). In recent years, local bifurcation techniques—most notably the Crandall--Rabinowitz theorem—have been widely employed in the analysis of fluid-type equations (see, e.g., \cite{castro2016uniformly,garcia2023global}). Obtaining global information on the branches that emanate from the bifurcation points is typically more delicate, as it requires a continuation principle together with \emph{a priori} bounds and a suitable topological degree. In our setting, we choose to work within the Fredholm framework via the Fitzpatrick--Pejsachowicz--Rabier degree. This approach is well adapted to the analytic structure of \eqref{Eq1.2} and yields a convenient global alternative, from which we derive global information on the continua of solutions bifurcating from the trivial branch.

This paper is organized as follows. In Section \ref{Se2} we establish the local and global well-posedness of \eqref{Eq3.1} for initial data in $H^{3}(\TT)$. In Section \ref{Se3} we analyze the local structure of the bifurcation diagram in Figure \ref{F1} using the Crandall--Rabinowitz theorem, and we determine the bifurcation direction of the branches emanating from the bifurcation points. Section \ref{Se4} is devoted to deriving global \emph{a priori} bounds for solutions of \eqref{Eq1.2}, which are crucial for the application of the global bifurcation theory. In Section \ref{Se5} we study the global behaviour of the connected component of solutions bifurcating from the first bifurcation point, and we prove the existence of a nontrivial steady state of \eqref{Eq1.1} for $\varepsilon\in(2^{r-s},1)$. Finally, Section \ref{Se6} describes the numerical methods used to compute the bifurcation diagram in Figure \ref{F1} and the solution profiles shown in Figure \ref{F2}.

\subsection{Notation}
Given $s > 0$, along this paper we work with the Sobolev spaces of periodic functions
\begin{align*}
	H^{s}(\mathbb{T}) :  = \left\{u(x)=\sum_{n\in\Z}\widehat{u}(n) \frac{e^{inx}}{\sqrt{2\pi}} \in L^{2}(\mathbb{T}) \, : \, \overline{\widehat{u}(n)}=\widehat{u}(-n), \; \; \sum_{n\in\Z} |n|^{2s} |\widehat{u}(n)|^{2} < +\infty\right\},
\end{align*}
endowed with the norm
$$\|u\|_{H^{s}}:=\sqrt{\|u\|_{L^{2}}^{2}+\|u\|_{\dot{H}^{s}}^{2}}, \quad \|u\|_{\dot{H}^{s}}:=\|\L^{s}u\|_{L^{2}}=\left(\sum_{n\in\Z}|n|^{2s}|\widehat{u}(n)|^{2}\right)^{\frac{1}{2}}.$$
Recall that $H^{s}(\mathbb{T})$ is a Hilbert space with inner product
$$
(u,v)_{H^{s}}:=(u,v)_{L^{2}}+(u,v)_{\dot{H}^{s}}, \quad (u,v)_{\dot{H}^{s}}:=(\L^{\frac{s}{2}}u, \L^{\frac{s}{2}}u)_{L^2} = \sum_{n\in\Z}|n|^{2s} \widehat{u}(n) \overline{\widehat{v}(n)}.
$$
For each $s > 0$, we will denote by $H_{0}^{s}(\mathbb{T})$ the subspace of $H^{s}(\mathbb{T})$ consisted on odd functions. Note that $H^{s}_{0}(\mathbb{T})$ is a Banach space as it is a closed subspace of $H^{s}(\mathbb{T})$. Analogously, we define $L^{q}_{0}(\mathbb{T})\subset L^{q}(\mathbb{T})$ for each $q\geq 1$ and $\mc{C}^{r}_{0}(\mathbb{T})\subset \mc{C}^{r}(\mathbb{T})$ for each $r\geq 0$.

Moreover, given a pair $(U,V)$ of real Banach spaces, the space of linear bounded operators $T:U\to V$ is denoted by $L(U,V)$. We denote by $GL(U,V)$ the space of topological isomorphisms.
Given $T\in L(U,V)$, we denote by $N[T]$ and $R[T]$, the kernel and the range of $T$, respectively. Finally, $\Phi_{0}(U,V)$ stands for the set of Fredholm operators of index zero and $\Phi_{0}(U):=\Phi_{0}(U,U)$.

\section{The evolution problem}\label{Se2}
In this section we establish the following result:
\begin{theorem}
Let us fix $\varepsilon > 0$ and
\[
u_0\in H^3(\mathbb{T}).
\]
Then, there exists a unique globally defined solution
\[
u(x,t)\in \mc{C}([0,T],H^3(\mathbb{T})),\;\forall T>0.
\]
\end{theorem}
\begin{proof}
Although this specific result is not contained in the already published literature, it can be obtained from somehow minor mofications of previous results. For this reason we only sketch the proof. The global existence for $0\leq r<1$ is contained in \cite{granero2015nonlocal}. Using the very same ideas one can recover the case $-1\leq r<0$.
\end{proof}

Once the global existence of solution has been established, we observe that for $0<\varepsilon<1$, the linear contribution coming from $\Lambda^r u$ makes the homogeneous steady state $U(x)\equiv 0 $ unstable for large scale perturbations. In fact, one can prove that
$$
\frac{1}{2}\frac{d}{dt}\|u(t)\|_{L^2}^2+\varepsilon\|u(t)\|_{H^{s/2}}^2=\|u(t)\|_{H^{r/2}}^2,
$$
which for initial data with compact support in low frequencies implies
$$
\frac{d}{dt}\|u(t)\|_{L^2}^2\geq0.
$$
Thus, the question of what is the behavior of such globally defined initial data remains.

\section{Existence of non-trivial steady states using local bifurcation}\label{Se3}
Here we recall the celebrated Crandall--Rabinowitz bifurcation theorem and some of its consequences, and we use it to analyze the local structure of the steady-state bifurcation diagram near the bifurcation points shown in Figure \ref{F1}. This theorem was stated and proved in \cite{CR, CRex}. Let $(U,V)$ be a pair of real Banach spaces and consider a map $\mf{F}:\R_{>0}\times U\to V$ of class $\mc{C}^{r}$, $r\geq 2$, satisfying the following assumptions:
\begin{enumerate}
	\item[(F1)] $\mathfrak{F}(\e,0)=0$ for every $\e>0$;
	\item[(F2)] $\mf{L}(\e):=\partial_{u}\mathfrak{F}(\e,0)\in\Phi_{0}(U,V)$ for each $\e>0$, \emph{i.e.} it is a Fredholm operator of zero index;
	\item[(F3)] $N[\mf{L}(\e_{0})]=\spann\{\varphi\}$ for some $\varphi\in U\backslash\{0\}$ and $\e_{0}>0$.
\end{enumerate}
The theorem reads as follows. For the involved notation and concepts we refer the reader to Appendix \ref{A1}.
\begin{theorem}
	\label{Th3.1} 
	Let $r\geq 2$, $\mathfrak{F}\in\mc{C}^{r}(\R_{>0}\times U,V)$ be a map satisfying conditions {\rm (F1)--(F3)} and $\e_{0}\in\Sigma(\mf{L})$ be a $1$-transversal eigenvalue of $\mf{L}$, that is, 
	$$\mf{L}_{1}(N[\mf{L}_{0}])\oplus R[\mf{L}_{0}]=V.$$
	Let $Y\subset U$ be a subspace	such that
	$N[\mathfrak{L}_0] \oplus Y = U$.
	Then, there exist $\d>0$ and two maps of class $\mc{C}^{r-1}$,
	$$
	\Omega: (-\d,\d) \longrightarrow \R_{>0}, \qquad \G: (-\d,\d)\longrightarrow Y,
	$$
	such that $\Omega(0)=\e_0$, $\G(0)=0$,
	and for each $s\in(-\d,\d)$,
	\begin{equation*}
		\mathfrak{F}(\Omega(s),u(s))=0, \quad u(s):= s(\v+\G(s)).
	\end{equation*}
	Moreover,  $\r >0$ exists such that if $\mathfrak{F}(\e,u)=0$ and
	$(\e,u)\in B_\r(\e_0,0)$, then either $u = 0$ or
	$(\e,u)=(\Omega(s),u(s))$ for some $s\in(-\d,\d)$.
\end{theorem}
Furthermore, we can also compute the bifurcation direction of the curve $s\mapsto (\O(s),u(s))$ in terms of the derivatives of $\mf{F}$. Let $\varphi^{\ast}\in V^{\ast}$ such that
$$
R[\mf{L}(\e_0)]=\{v\in V \; : \; \langle v, \varphi^{\ast}\rangle =0 \},
$$
where $\langle \cdot, \cdot \rangle : V\times V^{\ast}\to\R$ denotes the duality pairing on $V$. The following result can be found, for instance, in \cite{Ki}.

\begin{theorem}
	\label{Th3.2}
	Under the assumptions of Theorem \ref{Th3.1}, the following identity holds:
	\begin{align*}
		\dot\Omega(0)=-\frac{1}{2}\frac{\langle \partial^{2}_{uu}\mf{F}(\e_{0},0)[\varphi,\varphi],\varphi^{\ast}\rangle}{\langle \partial^{2}_{\e u}\mf{F}(\e_{0},0)[\varphi],\varphi^{\ast}\rangle}.
	\end{align*}
	Moreover, if $r\geq 3$ and $\partial^{2}_{uu}\mf{F}(\e_{0},0)[\varphi,\varphi]\in R[\mf{L}_{0}]$, then
	\begin{equation*}
		\ddot{\Omega}(0)=-\frac{1}{3}\frac{\langle\partial^{3}_{uuu}\mf{F}(\e_{0},0)[\varphi,\varphi,\varphi],\varphi^{\ast}\rangle+3\langle \partial^{2}_{uu}\mf{F}(\e_0,0)[\varphi,\phi],\varphi^{\ast}\rangle}{\langle \partial^{2}_{\e u}\mf{F}(\e_{0},0)[\varphi],\varphi^{\ast}\rangle},
	\end{equation*}
	where $\phi\in U$ is any vector satisfying $\partial^{2}_{uu}\mf{F}(\e_0,0)[\varphi,\varphi]+\partial_{u}\mf{F}(\e_0,0)[\phi]=0$.
\end{theorem}

In the next of this section we are going to check whether the previous abstract theorems can be applied to our setting. Let us consider the time independent analog of \eqref{Eq1.1}
\begin{equation}
	\label{Eq3.1}
	uu_{x}= \L^{r}u-\varepsilon \L^{s}u, \quad x\in \mathbb{T}, \;\; u\in H^{s}_{0}(\mathbb{T}),
\end{equation}
where $r<s$ and $s > 1$. This is a nonlinear and nonlocal elliptic equation. A \textit{solution} of \eqref{Eq3.1} is a function $u\in H^{s}_{0}(\mathbb{T})$ that satisfies \eqref{Eq3.1} pointwise almost everywhere. The solutions of the equation \eqref{Eq3.1} can be rewritten as the zeros of the following nonlinear operator
\begin{equation}
	\label{Eq3.2}
	\mf{F}:\R_{>0}\times H_{0}^{s}(\mathbb{T})\longrightarrow L_{0}^{2}(\mathbb{T}), \quad \mf{F}(\e,u)=\L^{r}u-\varepsilon \L^{s}u-u u_{x}.
\end{equation}
In fact, the previous nonlinear operator satisfies the following result.
\begin{lemma}\label{Le3.3}
	Let $(\e, u)\in \R_{>0}\times H_{0}^{s}(\mathbb{T})$. Then, 
	\begin{equation}
		\label{Eq3.3}
		\mf{F}(\e,u)=\L^{r}u-\varepsilon \L^{s}u-u u_{x}\in L^{2}_{0}(\mathbb{T}).
	\end{equation}
	Now we have to check that the hypotheses of Theorem \ref{Th3.1}	hold for our operator $\mf{F}$. It is clear that $\mf{F}(\varepsilon,u)$ satisfies {\rm(F1)}. We continue noticing that the operator $\mf{F}$ satisfies $\mf{F}\in\mc{C}^{\infty}(\R_{>0}\times  H_{0}^{s}(\mathbb{T}), L_{0}^{2}(\mathbb{T}))$. Moreover, its derivatives are given by
	\begin{itemize}
		\item[{\rm (i)}] For all $(\e,u)\in\R_{>0}\times  H_{0}^{s}(\mathbb{T})$,
		\begin{align}
			\label{Eq3.4}
			& \partial_{u}\mf{F}(\e,u): H_{0}^{s}(\mathbb{T}) \longrightarrow L_{0}^{2}(\mathbb{T}), \\
			& \partial_{u}\mf{F}(\e,u)[v]=\L^{r}v-\e \L^{s}v - (u_{x}v+uv_{x}). \nonumber
		\end{align}
		\item[{\rm (ii)}] For all $(\e,u)\in\R_{>0}\times  H_{0}^{s}(\mathbb{T})$, 
		\begin{align*}
			& \partial_{uu}^{2}\mf{F}(\l,u):H_{0}^{s}(\mathbb{T})\times H_{0}^{s}(\mathbb{T}) \longrightarrow L_{0}^{2}(\mathbb{T}), \\
			& \partial_{uu}^{2}\mf{F}(\e,u)[v_1,v_2] = -((v_{1})_{x}v_{2}+(v_{2})_{x}v_{1}).
		\end{align*}
		\item[{\rm (iii)}] Let $k\geq 3$. Then, for all $(\e,u)\in\R_{>0}\times  H_{0}^{s}(\mathbb{T})$, 
		\begin{align*}
			& \partial_{u}^{k}\mf{F}(\e,u):H_{0}^{s}(\mathbb{T})\times \overset{(k)}{\cdots} \times H_{0}^{s}(\mathbb{T}) \longrightarrow L_{0}^{2}(\mathbb{T}), \\
			& \partial_{u}^{k}\mf{F}(\e,u)[v_1,\dots,v_{k}] =0.
		\end{align*}
	\end{itemize}
\end{lemma}
\begin{proof}
	The proof of the first claim is straightforward and, for the sake of brevity, we omit it. 	Let us now prove the continuity of the operator $\mf{F}$. Let $\{(\e_{n},u_{n})\}_{n\in\N}\subset \R_{>0}\times H_{0}^{s}(\mathbb{T})$ be a sequence satisfying 
	$$\lim_{n\to+\infty}\e_n = \e_0>0 \quad \text{ and } \quad
	\lim_{n\to+\infty}u_{n}=u_{0} \quad \text{ in } \; \; H_{0}^{s}(\mathbb{T}).$$
	Choose $N\in\N$ such that $|\e_{n}-\e_{0}|<1$ and $\|u_{n}-u_{0}\|_{H^{s}}<1$ for each $n\geq N$. In particular, this implies that
	\begin{equation}
		\label{Eq3.5}
		|\e_{n}|<1+|\e_0| \;\; \text{and} \;\; \|u_{n}\|_{H^{s}}<1+\|u_{0}\|_{H^{s}}, \quad n\geq N.
	\end{equation}
	We rewrite the difference of the corresponding operator as
	\begin{align*}
		& \mf{F}(\e_{n},u_{n})-\mf{F}(\e_{0},u_{0}) \\ 
		&=\L^{r}(u_{n}-u_{0})-(\e_{n}-\e_{0})\L^{s}u_{0} -\e_{n}\L^{s}(u_{n}-u_{0})- (u_{n}-u_{0})(u_{n})_{x}-((u_{n})_{x}-(u_0)_{x})u_0. \nonumber
	\end{align*}
	Fixing $n\geq N$ and taking the $L^{2}$-norm, we obtain
	\begin{align*}
		\|\mf{F}(\e_{n},u_{n})-\mf{F}(\e_{0},u_{0})\|_{L^{2}} 
		\leq C \left( \|u_{n}-u_{0}\|_{H^{s}}+|\e_{n}-\e_{0}| \right),
	\end{align*}
	where we have used inequality \eqref{Eq3.5}. Therefore,
	$$\lim_{n\to+\infty}\|\mf{F}(\e_{n},u_{n})-\mf{F}(\e_{0},u_{0})\|_{L^{2}}=0.$$
	Hence, this implies that $\mf{F}$ is continuous.
	\par Let us prove that $\mf{F}\in\mc{C}^{\infty}$. It is enough to prove $\mf{F}\in\mc{C}^{1}$ as the proof of the higher differentiability relies in the same techniques. We start by proving \eqref{Eq3.4}. Let $(\e,u)\in\R_{>0}\times H^{s}_{0}(\mathbb{T})$ and $v\in H^{s}_{0}(\mathbb{T})$. A simple computation yields
	\begin{align*}
		\mf{F}(\e,u+v)-\mf{F}(\e,u)- \L^{r}v+\e \L^{s}v+(u_{x}v+uv_{x}) = -vv_{x}.
	\end{align*}
	Consequently, a direct bound on the $L^{2}$-norm yields
	\begin{align*}
		\|\mf{F}(\e,u+v)-\mf{F}(\e,u)- \L^{r}v+\e \L^{s}v+(u_{x}v+uv_{x}) \|_{L^{2}} = \|vv_{x}\|_{L^{2}}\leq \|v\|_{L^{\infty}}\|v_{x}\|_{L^{2}}\leq C \|v\|_{H^{s}}^{2}.
	\end{align*}
	Then, by the definition of differentiability,
	\begin{align*}
		\lim_{v\to 0}\frac{\|\mf{F}(\e,u+v)-\mf{F}(\e,u)- \L^{r}v+\e \L^{s}v+(u_{x}v+uv_{x}) \|_{L^{2}}}{\|v\|_{H^{s}}} 
		\leq C \lim_{v\to 0}\|v\|_{H^{s}}=0.
	\end{align*}
	This proves \eqref{Eq3.4}. The continuity of $\partial_{u}\mf{F} : \R_{>0}\times H^{s}_{0}(\mathbb{T})\to L(H^{s}_{0}(\mathbb{T}), L^{2}_{0}(\mathbb{T}))$ is proven in the same way we proved the continuity of $\mf{F}$. We conclude the result using the same ideas as before.
\end{proof}

The next result is of capital importance in order to apply bifurcation theory for Fredholm operators. It states that the linearization of $\mf{F}$ with respect to the variable $u$ at every point is a Fredholm operator of index zero.

\begin{lemma}
	For each $(\varepsilon,u)\in\R_{>0}\times H_{0}^{s}(\mathbb{T})$, the linear operator
	\begin{equation*}
		\partial_{u}\mathfrak{F}(\varepsilon,u): H_{0}^{s}(\mathbb{T}) \longrightarrow L_{0}^{2}(\mathbb{T}), 
		\qquad 
		\partial_{u}\mathfrak{F}(\varepsilon,u)[v]
		=
		\L^{r}v-\varepsilon \L^{s}v -(u v_{x}+v u_{x}),
	\end{equation*}
	is Fredholm of index zero.
\end{lemma}

\begin{proof}
	Noting that $\L^{s}: H^{s}_{0}(\TT)\to L^{2}_{0}(\TT)$ is a topological isomorphism with inverse $\L^{-s}:L^{2}_{0}(\TT)\to H^{s}_{0}(\TT)$, we can define the conjugated operator
	\[
	T_{\varepsilon,u}:=\L^{-s} \partial_{u}\mathfrak{F}(\varepsilon,u) :H^s_0(\mathbb{T})\longrightarrow H^s_0(\mathbb{T}).
	\]
	A direct computation gives, for every $v\in H^s_0(\mathbb{T})$,
	\begin{equation*}
		T_{\varepsilon,u}v
		=
		\L^{r-s}v-\varepsilon v-\L^{-s}(u v_x+v u_x),
	\end{equation*}
	hence
	\begin{equation*}
		T_{\varepsilon,u}
		=
		-\varepsilon I + K_{\varepsilon,u},
		\qquad
		K_{\varepsilon,u}
		:=
		\L^{r-s}-\L^{-s}(u\partial_x+u_x).
	\end{equation*}
	We will prove that $K_{\varepsilon,u}:H^s_0(\mathbb{T})\to H^s_0(\mathbb{T})$ is compact. Since $-\varepsilon I$ is an isomorphism on $H^s_0(\mathbb{T})$ for every $\varepsilon>0$, it will follow that $T_{\varepsilon,u}$ is Fredholm of index $0$, and then the same will hold for $\partial_{u}\mathfrak{F}(\varepsilon,u)$ because $\partial_{u}\mathfrak{F}(\varepsilon,u)=\L^s T_{\varepsilon,u}$ is the composition of $T_{\varepsilon,u}$ with an isomorphism.
	
	Let us show that $\L^{r-s}:H^s_0(\mathbb{T})\to H^s_0(\TT)$ is compact.
	Since $r<s$, we have $r-s<0$. Thus $\L^{r-s}=\L^{-(s-r)}$ gains $\delta:=s-r>0$ derivatives:
	\[
	\L^{r-s}:H^s_0(\mathbb{T})\longrightarrow H^{s+\delta}_0(\mathbb{T})
	=
	H^{2s-r}_0(\mathbb{T}).
	\]
	Moreover, because $\delta>0$, the Sobolev embedding $H^{s+\delta}(\mathbb{T})\hookrightarrow H^s(\mathbb{T})$
	is compact. Therefore the composition
	\[
	H^s_0(\mathbb{T})
	\xrightarrow{\L^{r-s}}
	H^{s+\delta}_0(\mathbb{T})
	\hookrightarrow
	H^s_0(\mathbb{T})
	\]
	is compact, i.e. $\L^{r-s}:H^s_0(\mathbb{T})\to H^s_0(\mathbb{T})$ is compact.
	
	Finally, we prove that $\L^{-s}(u\partial_x+u_x): H^s_0(\mathbb{T}) \to H^s_0(\mathbb{T})$ is compact.
	We use the standard Sobolev product bound on $\mathbb{T}$ in Fourier variables: since $s>1$,
	there exists $C>0$ such that for all $f\in H^s(\mathbb{T})$ and $g\in H^{s-1}(\mathbb{T})$,
	\begin{equation}\label{Eq3.6}
		\|fg\|_{H^{s-1}}
		\le C\,\|f\|_{H^s}\,\|g\|_{H^{s-1}}.
	\end{equation}
	Applying \eqref{Eq3.6} with $f=u$ and $g=v_x$, we obtain
	\[
	\|u v_x\|_{H^{s-1}}
	\le C\,\|u\|_{H^s}\,\|v\|_{H^s}.
	\]
	Similarly, since $u_x\in H^{s-1}$ and $v\in H^s$, another application of \eqref{Eq3.6} yields
	\[
	\|v u_x\|_{H^{s-1}}
	\le C\,\|v\|_{H^s}\,\|u_x\|_{H^{s-1}}
	\le C\,\|u\|_{H^s}\,\|v\|_{H^s}.
	\]
	Therefore,
	\begin{equation*}
		\|u v_x+v u_x\|_{H^{s-1}}
		\le C\,\|u\|_{H^s}\,\|v\|_{H^s},
		\qquad \forall\,u,v\in H^s(\mathbb{T}),
	\end{equation*}
	which shows that the first--order operator $B_u(v):=u v_x+v u_x$ is bounded
	\[
	B_u:H^s_0(\mathbb{T})\longrightarrow H^{s-1}_0(\mathbb{T}).
	\]
	Now we apply the smoothing operator $\L^{-s}$, which gains $s$ derivatives:
	\[
	\L^{-s}:H^{s-1}_0(\mathbb{T})\longrightarrow H^{2s-1}_0(\mathbb{T}).
	\]
	Hence the composition
	\[
	H^s_0(\mathbb{T})
	\xrightarrow{\,B_u\,}
	H^{s-1}_0(\mathbb{T})
	\xrightarrow{\L^{-s}}
	H^{2s-1}_0(\mathbb{T})
	\]
	is bounded. Since $s>1$ we have $2s-1>s$, and thus the embedding $H^{2s-1}_0(\mathbb{T})\hookrightarrow H^s_0(\mathbb{T})$ is compact. Therefore the operator
	\[
	\L^{-s}(u\partial_x+u_x):H^s_0(\mathbb{T})\longrightarrow H^s_0(\mathbb{T})
	\]
	is compact. This completes the proof.
\end{proof}

Thus, we have proved that our operator $\mathfrak{F}(\varepsilon,u)$ satisfies (F2). 

Now we turn our attention to the hypothesis (F3). We proceed with the spectral theoretic study of the linearization of the operator $\mf{F}$ on the \textit{trivial branch} $\mc{T}$ defined by 
$$
\mc{T}:=\{(\e,u)\in \R_{>0}\times H^{s}_{0}(\mathbb{T}) \; : \; u=0 \}.
$$
The linearization of $\mf{F}$ on $\mc{T}$ is given by the family of operators $\mf{L}(\e):= \partial_{u}\mf{F}(\e,0)$, $\e>0$, given explicitly by
$$
\mf{L}:\R_{>0}\longrightarrow\Phi_{0}(H^{s}_{0}(\mathbb{T}),L^{2}_{0}(\mathbb{T})), 
\quad \mf{L}(\e)[v] = \L^{r}v - \e \L^{s}v.
$$

Let us recall that the \textit{generalized spectrum} of $\mf{L}$ is given by
$$\Sigma(\mf{L}):=\{\e\in \R_{>0} : \, \mf{L}(\e)\notin GL\left(H^{s}_{0}(\mathbb{T}),L^{2}_{0}(\mathbb{T})\right)\}.$$
The next lemma describes the set $\Sigma(\mf{L})$ together with some other properties of the linearized operator.
\begin{lemma}
	For each $\e>0$, $\mf{L}(\e): H^{s}_{0}(\mathbb{T})\subset L^{2}_{0}(\mathbb{T}) \to L^{2}_{0}(\mathbb{T})$ is a self-adjoint operator.
	The generalized spectrum of the family $\mf{L}(\e)$ is given by
	\begin{equation*}
		\Sigma(\mf{L})=\left\{\s_k=k^{r-s} \, : \, k\in\N\right\}.
	\end{equation*}
	Moreover, they are ordered as
	\begin{equation*}
		0<\cdots<k^{r-s}<\cdots < 2^{r-s}< 1,
	\end{equation*}
	and it holds that
	\begin{align*}
		N[\mf{L}(k^{r-s})]=\spann\left\{\sin(kx)\right\}, \quad k\in \N.
	\end{align*}
	Furthermore, for each $k\in\N$, the generalized eigenvalue $\s_k$ is $1$-transversal and its generalized algebraic multiplicity is
	$$\chi[\mf{L},\s_{k}]=1.$$
\end{lemma}
\begin{proof}
	To show that $\mf{L}(\varepsilon)$ is selfadjoint is straightforward.
	
	Let $\e >0$ and take $v\in N[\mf{L}(\e)]$. Then, $\mf{L}(\e)[v]=0$ or equivalently $\L^{r}v-\e \L^{s}v=0$. In the Fourier side this becomes
	\begin{equation}
		\label{Eq3.7}
		\sum_{n\neq 0}\left(|n|^{r}-\e|n|^{s}\right)\widehat{u}(n) e^{i n x}=0.
	\end{equation}
	Hence, if $\e\neq k^{r-s}$ for each $k\in\N$, necessarily $\widehat{u}(n)=0$ for each $n\in\Z$ and this implies that $u\equiv0$. Hence $N[\mf{L}(\e)]=\{0\}$.
	On the one hand, if $\e= k^{r-s}$ for some $k\in \N$, then equation \eqref{Eq3.7} implies that 
	$$\widehat{u}(n)=0, \quad n\neq \pm k.$$
	Hence $u(x) =\widehat{u}(-k)e^{-ikx}+\widehat{u}(k)e^{i k x}=2i\widehat{u}(k)\sin(kx)$ and therefore 
	$$
	N[\mf{L}(k^{r-s})]=\spann\{\sin(kx)\}.
	$$ 
	Then $\s_k=k^{r-s}\in \Sigma(\mf{L})$ for each $k\in\N$. 
	
	Finally, let us show the transversality condition and the algebraic multiplicity. An elementary computation gives
	\begin{equation}
		\label{Eq3.8}
		\mf{L}_{1}(\e):=\frac{{\rm{d}}\mf{L}}{{\rm{d}}\e}(\e)= -\L^{s}, \quad \e>0.
	\end{equation}
	As $\mf{L}(\e)$ is a Fredholm operator, $R[\mf{L}(\e)]$ is closed in $L^{2}_{0}(\mathbb{T})$. Therefore, we deduce that in $L^{2}_{0}(\mathbb{T})$,
	$$
	R[\mf{L}(\e)]=\overline{R[\mf{L}(\e)]}=N[\mf{L}^{\ast}(\e)]^{\perp}=N[\mf{L}(\e)]^{\perp},
	$$
	where we have used that $\mf{L}(\e)$ is a self-adjoint operator on $L^{2}_{0}(\mathbb{T})$. Consequently, for each $k\in\N$,
	\begin{equation}
		\label{Eq3.9}
		R[\mf{L}(\s_{k})]=\left\{u\in L^{2}_{0}(\mathbb{T}) \, : \, \int_{\mathbb{T}}\sin(kx) u=0 \right\}.
	\end{equation}
	Hence, by \eqref{Eq3.8}, we obtain that
	\begin{align*}
		\mf{L}_{1}(\s_{k})[\sin(kx)] = -\L^{s}[\sin(kx)]=-k^{s}\sin(kx)\notin R[\mf{L}(\s_{k})], \quad k\in\N.
	\end{align*}
	Therefore, the transversality condition 
	$$
	\mf{L}_{ 1}(\s_{k})\left(N[\mf{L}(\s_{k})]\right)\oplus R[\mf{L}(\s_{k})]= L^{2}_{0}(\mathbb{T}),
	$$
	holds for each $k\in\N$. This implies that the eigenvalues $\s_{k}$ are $1$-transversal and using
	\begin{equation*}
		\chi[\mf{L},\l_0]=\dim N[\mf{L}_0],
	\end{equation*} we deduce that
	$$
	\chi[\mf{L},\s_{k}]=\dim N[\mf{L}(\s_{k})]= 1.
	$$
	The proof is concluded.
\end{proof}

Once we have checked that we can apply the Crandall--Rabinowitz Theorem, we study the local bifurcation of non-trivial solutions of equation \eqref{Eq3.1} from the points $(\s_k,0)\in \R_{>0}\times H^{s}_{0}(\mathbb{T})$. 

The \textit{set of non-trivial solutions} of $\mf{F}$ is defined by
\begin{equation*}
	\mf{S}:=\left[ \mf{F}^{-1}(0)\backslash \mc{T}\right]\cup \{(\e,0):\;\e\in \Sigma(\mf{L})\}\subset \R_{>0}\times H^{s}_{0}(\mathbb{T}).
\end{equation*}
Due to the continuity of the operator $\mf{F}$, see Lemma \ref{Le3.3}, the subset $\mf{S}$ is closed in $\R_{>0}\times H^{s}_{0}(\mathbb{T})$.
The main local result is the following.

\begin{theorem}
	\label{Th3.6}
	Let $k\in\N$. Then, the point $(\s_{k},0)\in \R_{>0} \times H^{s}_{0}(\mathbb{T})$ is a bifurcation point of the non-linearity
	$$\mf{F}:\R_{>0}\times H^{s}_{0}(\mathbb{T})\longrightarrow L^{2}_{0}(\mathbb{T}), \quad \mf{F}(\e,u):= \L^{r}u-\e \L^{s}u - uu_{x},$$
	from the trivial branch $\mc{T}$ to a connected component $\mathscr{C}_{k}$ of the set of non-trivial solutions $\mf{S}$. Let 
	$$Y_{k}:=\{u\in H^{s}_{0}(\mathbb{T}) \, : \, (u,\sin(kx))_{H^{2s}}=0 \}.$$
	Then, the following statements hold:
	\vspace{2pt}
	\begin{enumerate}
		\item[{\rm{(a)}}] {\rm\textbf{Existence:}} There exist $\d>0$ and two $\mc{C}^{\infty}$-maps
		\begin{equation*}
			\Omega_{k}: (-\d,\d) \longrightarrow \R, \quad \L_{k}(0)=\s_{k}, \qquad \Gamma_{k}: (-\d,\d)\longrightarrow Y_{k}, \quad  \Gamma_{k}(0)=0,
		\end{equation*}
		such that for each $t\in(-\d,\d)$,
		\begin{equation*}
			\mathfrak{F}(\Omega_{k}(t),u_{k}(t))=0, \quad u_{k}(s):= t(\sin(kx)+\Gamma_{k}(s)).
		\end{equation*}
		In other words, for some $\r>0$,
		$$\mathscr{C}_{k}\cap B_{\r}(\s_{k},0)=\{(\Omega_{k}(t), t(\sin(kx)+\Gamma_{k}(s))) \, : \, t\in(-\d,\d)\}.$$
		\vspace{1pt}
		\item[{\rm{(b)}}] {\rm\textbf{Uniqueness:}} There exists $\r >0$ such that if $\mathfrak{F}(\e,u)=0$ and
		$(\e,u)\in B_\r(\s_{k},0)\subset \R_{>0}\times H^{s}_{0}(\mathbb{T})$, then either $u = 0$ or
		$(\e,u)=(\Omega_{k}(t),u_{k}(t))$ for some $t\in(-\d,\d)$. Therefore, from $(\s_{k},0)$, there emanate precisely two branches of non-trivial solutions of equation \eqref{Eq3.1}. In other words,
		$$\mf{S}\cap B_{\rho}(\s_{k},0)=\mathscr{C}_{k}\cap B_{\rho}(\s_{k},0).$$
		Moreover, locally the solutions emanating from $(\s_{k},0)$ have exactly $2k$ zeros.
		\vspace{5pt}
		\item[{\rm{(c)}}] {\rm\textbf{Bifurcation direction:}} The bifurcation direction is given by
		\begin{equation*}
			\dot{\Omega}(0)=0, \quad \ddot{\Omega}(0)=\frac{k^{2-s-r}}{2^{r+1}(1-2^{s-r})}.
		\end{equation*}
		Then $\ddot{\Omega}(0)<0$ and hence we are dealing with a subcritical bifurcation. This implies that for a sufficiently small $\r>0$, if $(\e,u)\in B_\r(\s_{k},0)$ and $u\neq 0$, then necessarily $\e<\s_k$.
	\end{enumerate}
	
\end{theorem}

\begin{proof}
	After showing that $\mf{F}$ satisfies the hypothesis of Theorem \ref{Th3.1} satisfies the hypotheses of the Crandall--Rabinowitz Theorem in the previous Lemmas, the first part of the result, items (a) and (b), are a direct application of the Crandall--Rabinowitz Theorem \ref{Th3.1} applied to the nonlinearity $\mf{F}$. We proceed to prove item (c). By Theorem \ref{Th3.2}, we have that
	\begin{equation}
		\label{Eq3.10}
		\dot\Omega_{k}(0)=-\frac{1}{2}\frac{( \partial^{2}_{uu}\mf{F}(\s_{k},0)[\varphi_{k},\varphi_{k}],\varphi_{k}^{\ast})_{L^2}}{( \partial^{2}_{\e u}\mf{F}(\s_{k},0)[\varphi_{k}],\varphi_{k}^{\ast})_{L^2}},
	\end{equation}
	where $\varphi_{k}(x)=\sin(kx)$, $x\in\mathbb{T}$, and $\varphi^{\ast}_{k}\in L^{2}_{0}(\mathbb{T})$ is a function such that
	$$
	R[\mf{L}(\s_k)]=\{f\in L^{2}_{0}(\mathbb{T}) \; : \; (f,\varphi^{\ast}_{k})_{L^{2}}=0\}.
	$$
	By identity \eqref{Eq3.9}, we can choose $\varphi^{\ast}_{k} = \sin(kx)$. An standard computation using Lemma \ref{Le3.3} gives
	\begin{align*}
		\partial^{2}_{uu}\mf{F}(\s_{k},0)[\varphi_{k},\varphi_{k}]=-k\sin(2kx), \quad \partial^{2}_{\e u}\mathfrak{F}(\s_{k},0)[\varphi_{k}]=- k^{s}\sin(kx).
	\end{align*}
	Therefore, a direct substitution on identity \eqref{Eq3.10} yields
	\begin{equation*}
		\dot{\Omega}_{k}(0)=-\frac{1}{2}\frac{\left(- k\sin(2kx),\sin(kx)\right)_{L^{2}}}{\left( -k^{s}\sin(kx),\sin(kx)\right)_{L^{2}}}= -\frac{k^{1-s}}{2\pi}\int_{\mathbb{T}}\sin(2kx)\sin(kx)\; dx = 0.
	\end{equation*}
	This proves that $\dot{\Omega}_{k}(0)=0$. On the other hand, by Lemma \ref{Le3.3}, we get
	$$
	\partial^{3}_{uuu}\mf{F}(\s_{k},0)[\varphi_{k},\varphi_{k},\varphi_{k}]=0,
	$$
	therefore, by Theorem \ref{Th3.2}, 
	\begin{align*}
		\ddot\Omega_{k}(0)=-\frac{(\partial^{2}_{uu}\mf{F}(\s_{k},0)[\varphi_{k},\phi_{k}],\varphi_{k}^{\ast})_{L^2}}{( \partial^{2}_{\e u}\mf{F}(\s_{k},0)[\varphi_{k}],\varphi_{k}^{\ast})_{L^2}},
	\end{align*}
	where $\phi_{k}\in H^{s}_{0}(\mathbb{T})$ is any function satisfying
	$$
	\partial^{2}_{uu}\mf{F}(\s_{k},0)[\varphi_{k},\varphi_{k}]+\partial_{u}\mf{F}(\s_{k},0)[\phi_{k}]=0.
	$$
	This is equivalent to the non-homogeneous pseudo-differential equation
	\begin{equation}
		\label{Eq3.11}
		\L^{r}\phi_{k}-\s_{k}\L^{s}\phi_{k}=k\sin(2kx), \quad x\in\mathbb{T}.
	\end{equation}
	As we can rewrite
	$$k\sin(2kx)=  -\frac{k}{2i}e^{-2kxi} +\frac{k}{2i}e^{2kxi},$$
	expanding equation \eqref{Eq3.11} in Fourier series, we obtain
	$$
	\sum_{n\neq 0}\left(|n|^{r}-k^{r-s}|n|^{s}\right)\widehat{\phi_{k}}(n)\frac{e^{inx}}{\sqrt{2\pi}} = -\frac{k}{2i}e^{-2kxi} +\frac{k}{2i}e^{2kxi}.
	$$
	Comparing each summand, we deduce that
	\begin{align*}
		&	\widehat{\phi_{k}}(2k)=-\widehat{\phi_{k}}(-2k)=\frac{\sqrt{2\pi}}{2i}\frac{k^{1-r}}{2^{r}(1-2^{s-r})}, \\
		& \widehat{\phi_{k}}(n)=0, \quad n\neq -2k, -k, k, 2k.
	\end{align*}
	Therefore, the function $\phi_k$ is given by
	\begin{equation*}
		\phi_{k}(x) = \frac{2i}{\sqrt{2\pi}}\widehat{\phi_{k}}(k)\sin(kx) + \frac{k^{1-r}}{2^{r}(1-2^{s-r})}\sin(2kx).
	\end{equation*}
	We take $\widehat{\phi_{k}}(k)=0$. This explicit representation of $\phi_{k}$ gives
	\begin{align*}
		\partial^{2}_{uu}\mf{F}(\s_{k},0)[\varphi_{k},\phi_{k}] & = -(\varphi_{k})_{x}\phi_{k}-\varphi_{k}(\phi_{k})_{x} \\
		& =-\frac{k^{2-r}}{2^{r}(1-2^{s-r})}\left(\cos(kx)\sin(2kx)+2\cos(2kx)\sin(kx)\right),
	\end{align*}
	therefore,
	\begin{align*}
		&\left(\partial^{2}_{uu}\mf{F}(\s_{k},0)[\varphi_{k},\phi_{k}], \varphi_{k}\right)_{L^{2}}  \\
		& = -\frac{k^{2-r}}{2^{r}(1-2^{s-r})}\left(\frac{1}{2}\int_{\TT}\sin^{2}(2kx) \; dx+2\int_{\TT}\cos(2kx)\sin^{2}(kx) \; dx\right)\\ 
		& = \frac{\pi}{2}\frac{k^{2-r}}{2^{r}(1-2^{s-r})}=\frac{k^{2-r}\pi}{2^{r+1}(1-2^{s-r})}.
	\end{align*}
	Finally, from the identity
	$$
	\left(\partial^{2}_{\e u}\mf{F}(\s_{k},0)[\varphi_{k}],\varphi_{k}^{\ast}\right)_{L^{2}}=\left(-k^{s}\sin(kx), \sin(kx)\right)_{L^{2}}=-k^{s}\pi,
	$$
	we deduce that
	$$
	\ddot{\Omega}_{k}(0) =-\frac{(\partial^{2}_{uu}\mf{F}(\s_{k},0)[\varphi_{k},\phi_{k}],\varphi_{k}^{\ast})_{L^2}}{( \partial^{2}_{\e u}\mf{F}(\s_{k},0)[\varphi_{k}],\varphi_{k}^{\ast})_{L^2}}=\frac{k^{2-s-r}}{2^{r+1}(1-2^{s-r})}.
	$$
	This concludes the proof.
\end{proof}

Thus, roughly speaking we have proved the local structure (near the bifurcation points $(\s_{k},0)$, $k\in\N$) of the bifurcation diagram of Figure \ref{F1}.

Moreover, we prove that, locally near each bifurcation point $(\s_k,0)$, the connected component $\mathscr{C}_k$ consists of solutions with exactly $2k$ simple zeros. Our numerical continuation (see Section~\ref{B}) illustrates this structure and computes representative profiles for the case $r=\tfrac12$ and $s=\tfrac32$. Figure~\ref{F2} displays some families of steady states along the components $\mathscr{C}_k$ for $k=1,2,3,4$ as shown in Figure \ref{F1}. From left to right (blue, red, yellow and black), we show numerical profiles on $\mathscr{C}_1$ for $\varepsilon\in[0.1581,1]$, on $\mathscr{C}_2$ for $\varepsilon\in[0.1508,0.5]$, on $\mathscr{C}_3$ for $\varepsilon\in[0.1477,0.3333]$, and on $\mathscr{C}_4$ for $\varepsilon\in[0.1492,0.25]$.

\begin{center}
	\begin{figure}[h!]
		\includegraphics[scale=0.4]{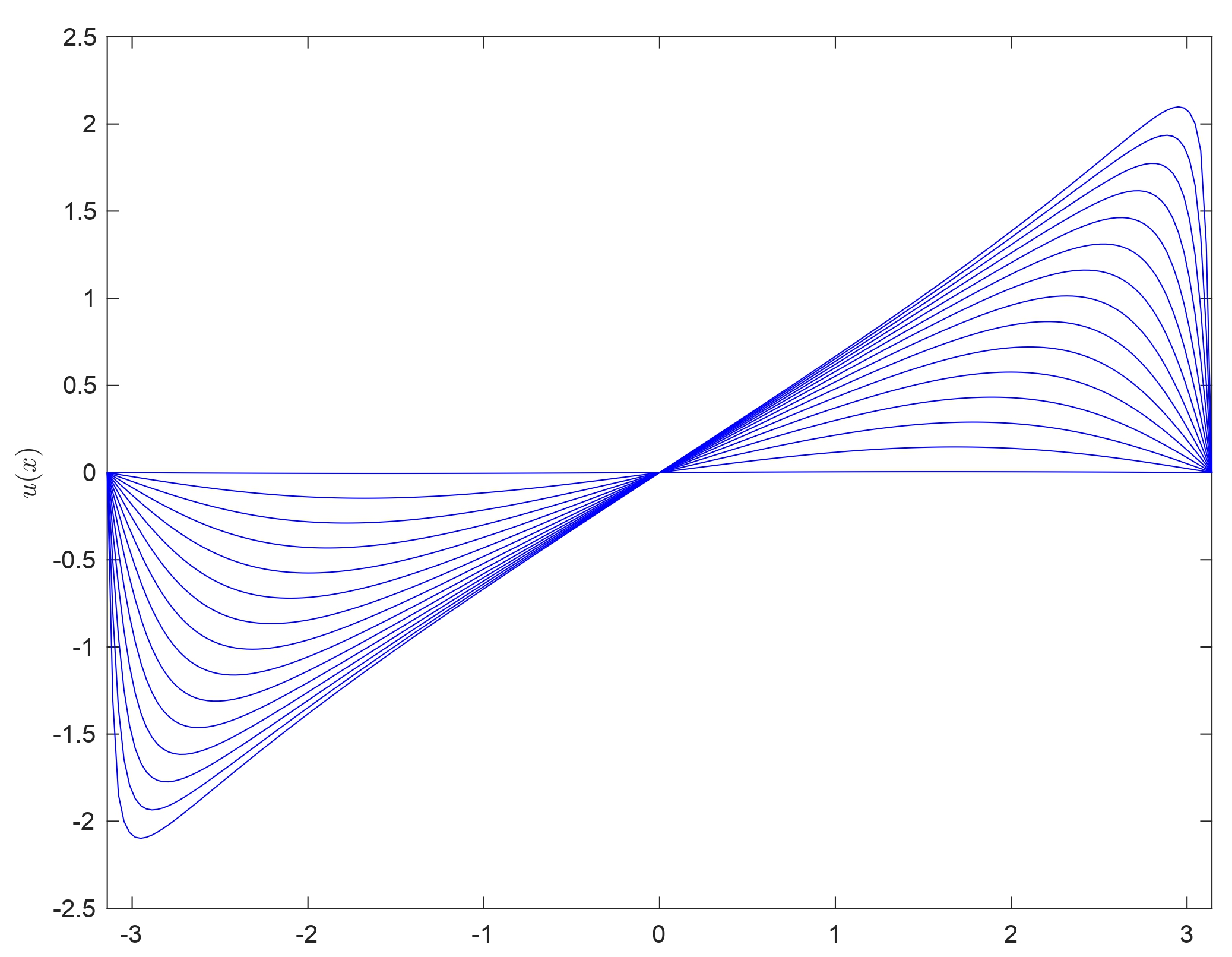}
		\includegraphics[scale=0.4]{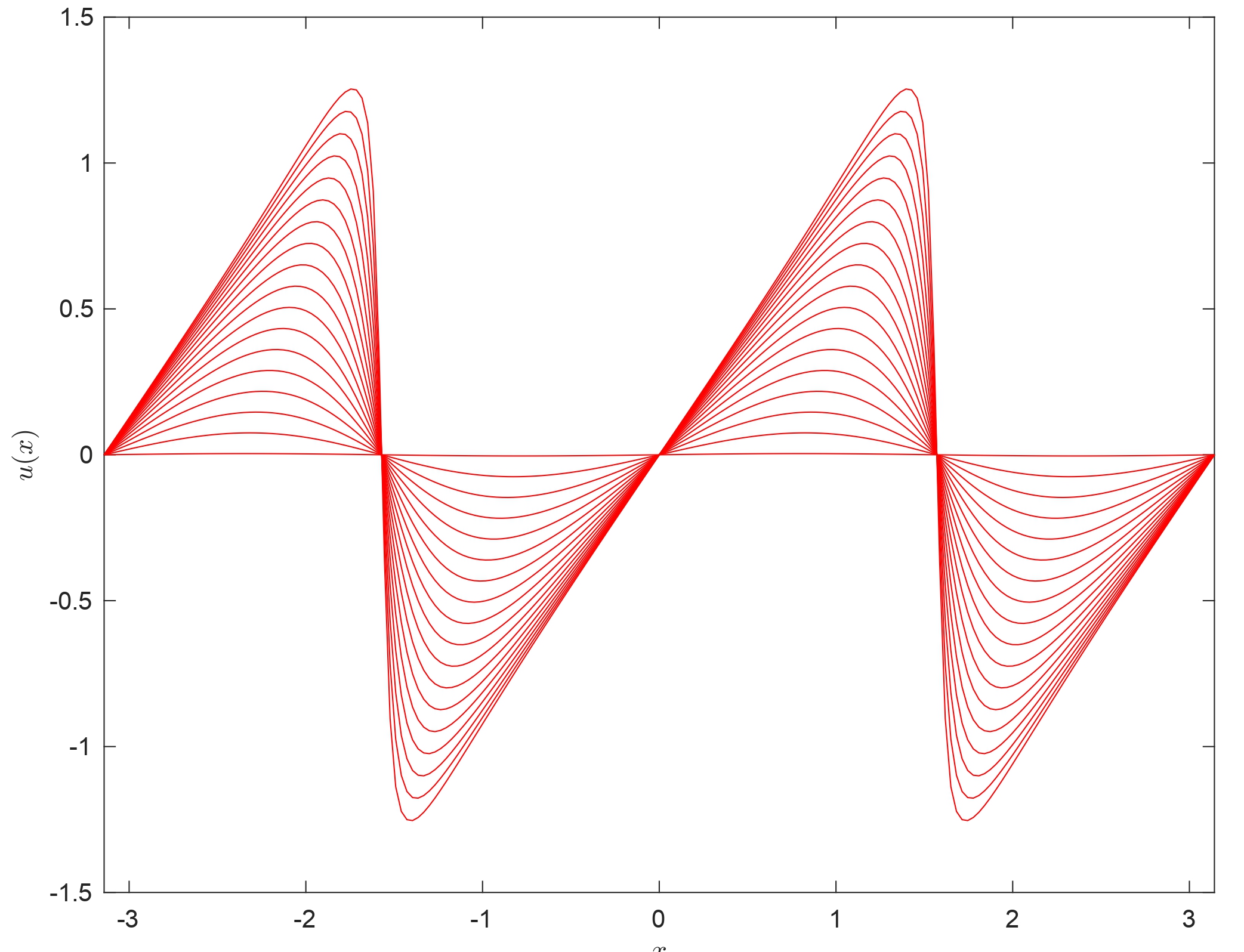}
		\includegraphics[scale=0.4]{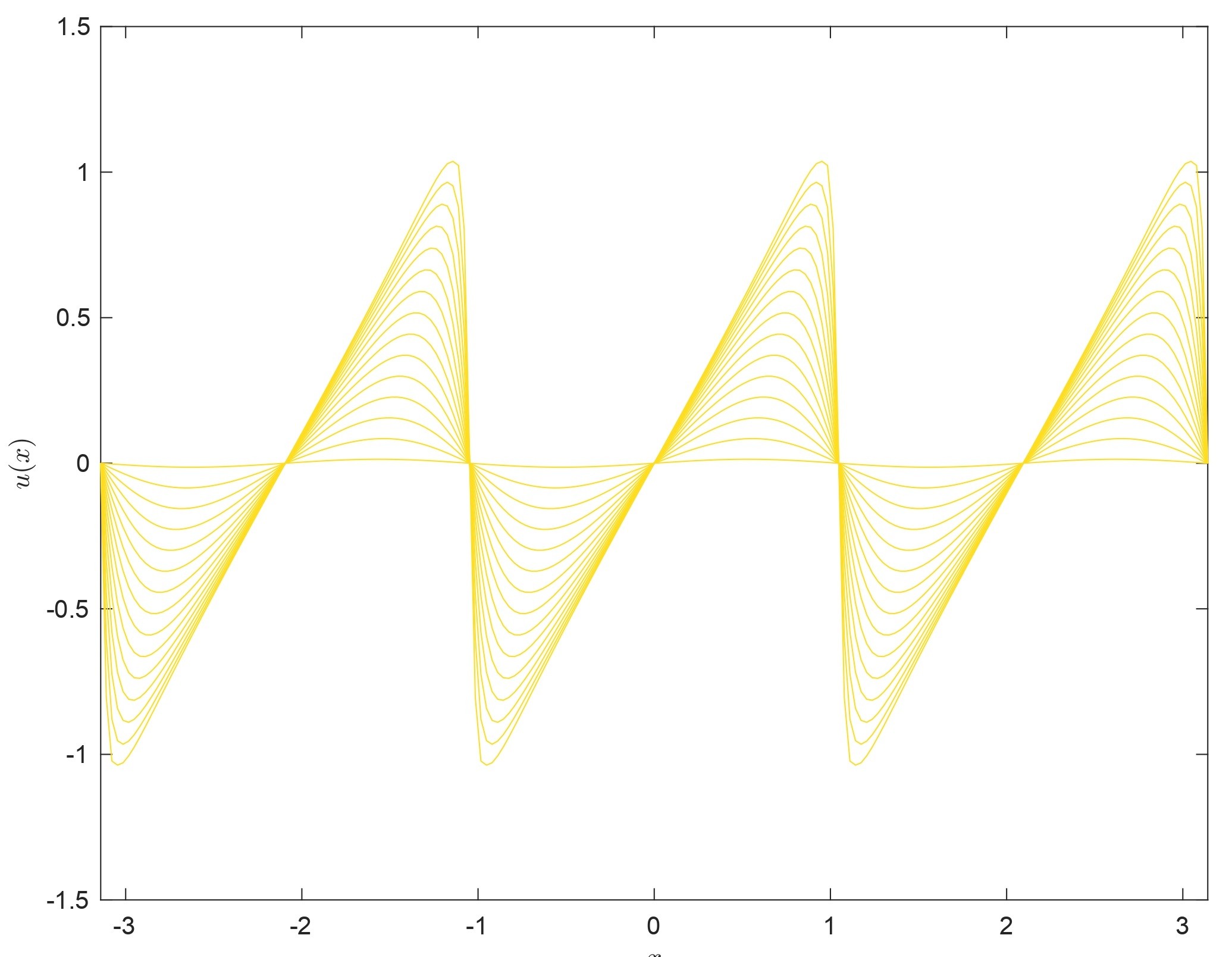}
		\includegraphics[scale=0.4]{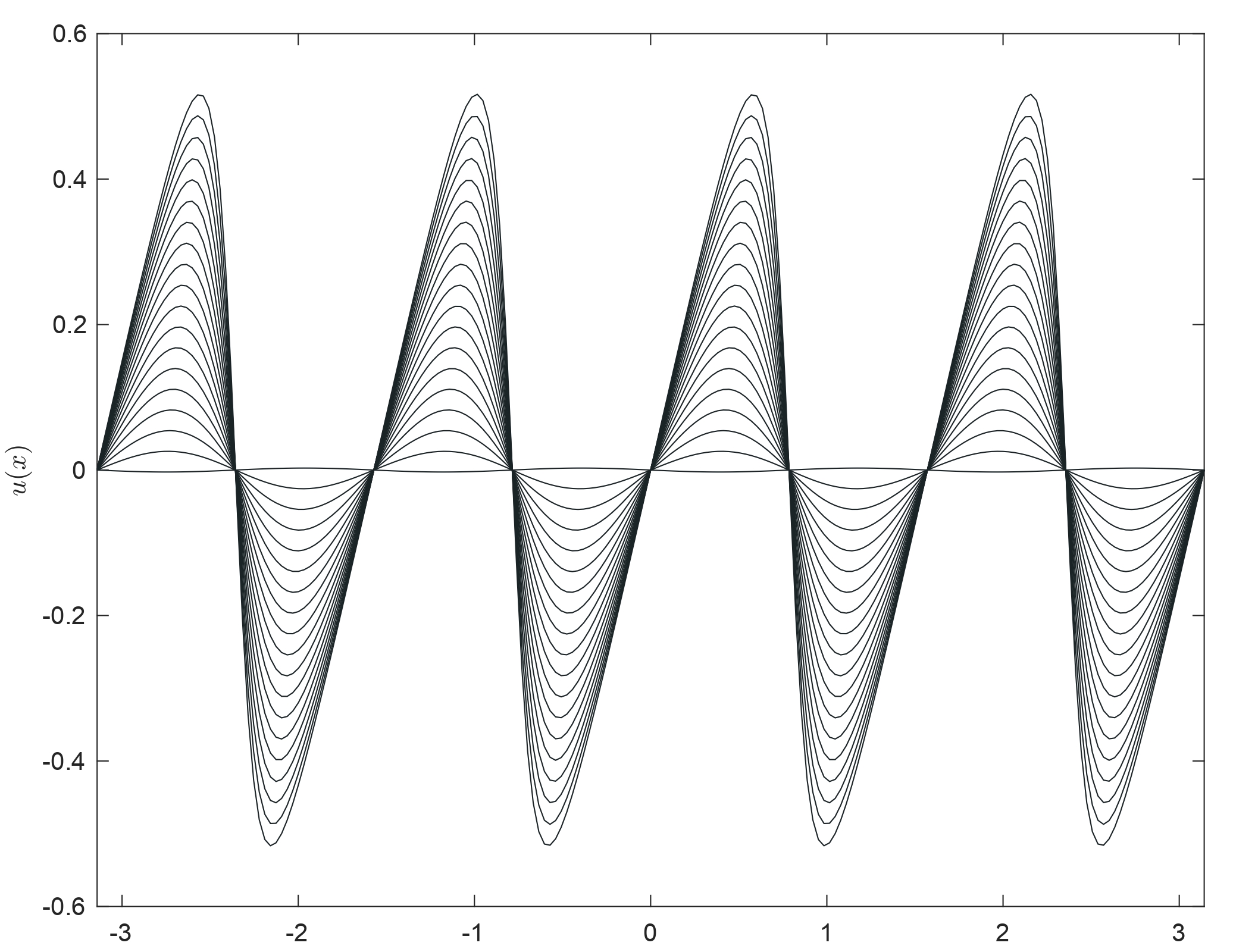}
		\caption{Profiles of some of the solutions (for $r=\tfrac{1}{2}$ and $s=\tfrac{3}{2}$) of the connected components $\mathscr{C}_{k}$ for $k=1,2,3,4$, respectively.}
		\label{F2}
	\end{figure}
\end{center}

\section{A priori estimates}\label{Se4}
In this section we collect some estimates for the solutions of \eqref{Eq1.2}.

Let us start with the linear version of \eqref{Eq1.1}. The next is a regularity result for the weak solutions of the corresponding linear pseudo--differential equation involving the Fourier multiplier operators $\L^{s}$ and $\L^{r}$. We omit the proof as it is straightforward.

\begin{lemma}
	\label{Le4.1}
	Consider $c\in L^{\infty}_{0}(\mathbb{T})$, $d\in L^{\infty}(\mathbb{T})$ an even function and $f\in L^{2}_{0}(\mathbb{T})$. Let $u\in H_{0}^{s/2}(\mathbb{T})$ be a weak solution of the linear problem
	\begin{equation*}
		\L^{r}u-\e\L^{s}u+c(x)u_{x}+d(x)u=f, \quad x\in\mathbb{T}.
	\end{equation*}
	Then, $u\in H^{s}_{0}(\TT)$ and the following estimate holds
	\begin{equation}
		\label{Eq4.1}
		\varepsilon\|u\|_{\dot{H}^{s}}\leq \|f\|_{L^{2}}+\|d\|_{L^\infty}\|u\|_{L^{2}}+\|u\|_{\dot{H}^{r}}+\|c\|_{L^\infty}\|u\|_{\dot{H}^{1}}.
	\end{equation}
\end{lemma}

The rest of this section is devoted to obtain global a priori bounds of solutions of \eqref{Eq1.2}. The following result establishes the range of $\e>0$ in which we may have non trivial solutions.
\begin{lemma}
	\label{Le4.2}
	Let $(\varepsilon,u)\in \R_{>0}\times H^{s}_{0}(\TT)$ be a solution of \eqref{Eq3.1} with $u\not\equiv 0$. Then:
	\begin{enumerate}
		\item[{\rm(a)}] $\varepsilon\in (0,1)$.
		\item[{\rm(b)}] The following identity holds:
		\begin{equation}
			\label{Eq4.2}
			 \frac{\|u\|_{\dot{H}^{r/2}}}{\|u\|_{\dot{H}^{s/2}}}=\sqrt{\varepsilon}.
		\end{equation}
	\item[{\rm(c)}] The following estimate holds:
	$$
	\varepsilon \|u\|_{\dot{H}^s}\leq \|u\|_{\dot{H}^r}+\|u\|_{\dot{H}^1}\|u\|_{L^\infty}.
	$$
	\end{enumerate}
\end{lemma}

\begin{proof}
	Let $(\varepsilon,u)\in \R_{>0}\times H^{s}_{0}(\TT)$ be a solution of \eqref{Eq3.1}. Multiplying the equation by $u$ and integrating over $\TT$, we obtain
	\begin{equation}
		\label{Eq4.3}
		\|u\|_{\dot{H}^{r/2}}^{2}-\varepsilon\|u\|_{\dot{H}^{s/2}}^{2}=0.
	\end{equation}
	Therefore, if $u\not\equiv 0$ is a nontrivial solution of \eqref{Eq3.3}, necessarily $\e>0$. For $u\not\equiv 0$, as $r<s$, due to the symmetry and Poincaré inequality, we always have
	\begin{equation}
	\label{Eq4.4}
	\|u\|_{\dot{H}^{r/2}} \leq \|u\|_{\dot{H}^{s/2}}.
	\end{equation}
	Therefore, combining this with the previous inequality \eqref{Eq4.3} we obtain $\varepsilon\leq 1$. Note that equality in \eqref{Eq4.4} can only occur for functions supported on a single frequency (that is $u(x)=k\sin x$), which are not solutions of \eqref{Eq3.1} for $k\neq 0$. Hence $\varepsilon < 1$. This proves item~(a). Item (b) follows directly by item (a) and \eqref{Eq4.3}. In order to prove item (c) we just multiply \eqref{Eq3.1} by $\Lambda^s u$ and integrate by parts.
\end{proof}
Now, we turn our attention to $H^{s}_{0}(\TT)$ a priori bounds of \eqref{Eq1.2}. We start by performing a blow-up type arguments required to obtain these estimates.
\begin{lemma}
	\label{Le4.3}
	Let $(\e,u_{n})\in (0,1)\times H^{s}_{0}(\TT)$ be a sequence of nontrivial solutions of equation \eqref{Eq3.3} such that 
	$$
	\lim_{n\to+\infty}\e_{n}=\e_{\ast} \;\; \text{and} \;\;  \lim_{n\to +\infty}\|u_{n}\|_{\dot{H}^{s/2}}=+\infty.
	$$
	Then, it holds that
	\begin{equation}
		\label{Eq4.5}
		\lim_{n\to+\infty}\frac{\|u_{n}\|_{L^{2}}}{\|u_{n}\|_{\dot{H}^{s/2}}}=0.
	\end{equation}
	Moreover, we also have
	\begin{equation}
		\label{Eq4.6}
		\lim_{n\to+\infty}\frac{\|u_{n}\|_{L^{\infty}}}{\|u_{n}\|_{\dot{H}^{s/2}}}=0.
	\end{equation}
\end{lemma}

\begin{proof}
	For each $n\in\mathbb N$ define
	\[
	v_n:=\frac{u_n}{\|u_n\|_{\dot H^{s/2}}}.
	\]
	Then $v_n\in H^{s}_{0}(\mathbb T)$, $\|v_n\|_{\dot H^{s/2}}=1$ for all $n$, and $v_n$ satisfies
	\begin{equation}\label{Eq4.7}
		v_n (v_n)_x
		=
		\frac{\L^{r}v_n-\varepsilon_n \L^{s}v_n}{\|u_n\|_{\dot H^{s/2}}}
		\qquad\text{in } \mathcal D'(\mathbb T).
	\end{equation}
	Since $r<s$ and $\|v_n\|_{\dot H^{s/2}}=1$, we have that
	\[
	\|\L^{r}v_n\|_{L^{2}}^{2}\leq 1
	\]
	Since $\varepsilon_n\to\varepsilon_\ast$, there exists $M>0$ such that $0<\varepsilon_n\le M$ for all $n$.
	Therefore,
	\begin{equation}\label{Eq4.8}
		\|\L^{r}v_n-\varepsilon_n\L^{s}v_n\|_{L^{2}}
		\le \|\L^{r}v_n\|_{L^{2}}+\varepsilon_n\|\L^{s}v_n\|_{L^{2}}
		\le C_0,
	\end{equation}
	with $C_0$ independent of $n$.
	Dividing \eqref{Eq4.8} by $\|u_n\|_{\dot H^{s/2}}\to\infty$ and using \eqref{Eq4.7}, we obtain
	\begin{equation}\label{Eq4.9}
		\|v_n (v_n)_x\|_{L^{2}}
		\le
		\frac{\|\L^{r}v_n-\varepsilon_n\L^{s}v_n\|_{L^{2}}}{\|u_n\|_{\dot H^{s/2}}}
		\le \frac{C_0}{\|u_n\|_{\dot H^{s/2}}}
		\longrightarrow 0.
	\end{equation}
	Since $(v_n)$ is bounded in $H^{s/2}(\mathbb T)$ and $s > 1$, the embedding
	$H^{s/2}(\mathbb T)\hookrightarrow L^{4}(\mathbb T)$ is compact (indeed $s/2>1/4$ in dimension one).
	Hence, up to a subsequence,
	\[
	v_n \rightharpoonup v_\ast \ \text{in } H^{s/2}(\mathbb T),
	\qquad
	v_n \to v_\ast \ \text{in } L^{4}(\mathbb T).
	\]
	In particular, $v_n^2\to v_\ast^2$ in $L^{2}(\mathbb T)$.
	
	Let $\varphi\in \mc{C}^\infty(\mathbb T)$. Using periodicity and integration by parts,
	\[
	\int_{\mathbb T} v_n^2\,\varphi_x\,dx
	=
	-2\int_{\mathbb T} v_n (v_n)_x\,\varphi\,dx.
	\]
	By \eqref{Eq4.9}, the right-hand side tends to $0$ as $n\to\infty$ (since $\varphi\in L^\infty$).
	Passing to the limit on the left-hand side using $v_n^2\to v_\ast^2$ in $L^2$, we obtain
	\[
	\int_{\mathbb T} v_\ast^2\,\varphi_x\,dx=0
	\qquad\text{for all }\varphi\in \mc{C}^\infty(\mathbb T).
	\]
	Therefore $(v_\ast^2)_x=0$ in $\mathcal D'(\mathbb T)$, hence $v_\ast^2$ is constant on $\mathbb T$.
	Since each $v_n$ is odd with mean zero, the same holds for $v_\ast$; in particular $v_\ast$ cannot be a nonzero constant.
	Thus $v_\ast\equiv 0$.
	
	Consequently, the whole sequence satisfies $v_n\to0$ in $L^{2}(\mathbb T)$, and hence
	\[
	\frac{\|u_n\|_{L^{2}}}{\|u_n\|_{\dot H^{s/2}}}
	=
	\|v_n\|_{L^{2}}
	\longrightarrow 0,
	\]
	which proves \eqref{Eq4.5}.
	Since $s>1$, then $s/2>1/2$ and the embedding $H^{s/2}(\mathbb T)\hookrightarrow \mc{C}(\mathbb T)$ is compact.
	Hence, up to a subsequence, $v_n\to v_\ast$ in $\mc{C}(\mathbb T)$, and since $v_\ast\equiv 0$ we get
	$\|v_n\|_{L^\infty}\to0$.
	Therefore
	\[
	\frac{\|u_n\|_{L^\infty}}{\|u_n\|_{\dot H^{s/2}}}
	=
	\|v_n\|_{L^\infty}
	\longrightarrow 0,
	\]
	which proves \eqref{Eq4.6}.
\end{proof}

Thanks to Lemma \ref{Le4.3}, next we obtain $H^{s/2}$-estimates for the solutions of \eqref{Eq1.2}.

\begin{lemma}\label{Le4.4}
	Let $\delta\in(0,1)$. Then there exists a constant $C_1(\delta)>0$ such that every solution
	$(\varepsilon,u)\in(\delta,+\infty)\times H^{s}_{0}(\mathbb T)$ of \eqref{Eq3.3} satisfies
	\begin{equation}\label{Eq4.10}
		\|u\|_{H^{s/2}}\le C_{1}(\delta).
	\end{equation}
	Moreover, there exists $C_2(\delta)>0$ such that every such solution also satisfies
	\begin{equation}\label{Eq4.11}
		\|u\|_{L^{\infty}}\le C_{2}(\delta).
	\end{equation}
	In particular, the estimates are uniform for $\varepsilon\in[\delta,1)$, whereas for $\varepsilon\ge 1$ any solution is trivial.
\end{lemma}

\begin{proof}
	Fix $\delta\in(0,1)$ and let $(\varepsilon,u)\in(\delta,+\infty)\times H_0^s(\mathbb T)$ be a solution of \eqref{Eq3.3}.
	If $u\equiv 0$, the conclusion is trivial. Assume $u\not\equiv 0$.
	By Lemma~\ref{Le4.2}(a), any nontrivial solution satisfies $\varepsilon\in(0,1)$; hence, necessarily $\varepsilon\in(\delta,1)$.
	We claim that there exists $C(\delta)>0$ such that every solution $(\varepsilon,u)$ of \eqref{Eq3.3} with $\varepsilon\in[\delta,1)$ satisfies
	\begin{equation}\label{Eq4.12}
		\|u\|_{\dot H^{s/2}}\le C(\delta).
	\end{equation}
	Suppose by contradiction that \eqref{Eq4.12} fails. Then there exist solutions
	\[
	(\varepsilon_n,u_n)\in[\delta,1)\times H_0^s(\mathbb T),\qquad u_n\not\equiv0,
	\]
	such that
	\[
	\lim_{n\to+\infty}\|u_n\|_{\dot H^{s/2}}= +\infty.
	\]
	Since $(\varepsilon_n)$ is bounded in $[\delta,1)$, up to a subsequence we may assume
	\[
	\lim_{n\to+\infty}\varepsilon_n=\varepsilon_\ast\in[\delta,1).
	\]
	Therefore the hypotheses of Lemma~\ref{Le4.3} apply, and defining
	\[
	v_n:=\frac{u_n}{\|u_n\|_{\dot H^{s/2}}},
	\]
	we have $\|v_n\|_{\dot H^{s/2}}=1$ for all $n$ and
	\begin{equation}\label{Eq4.13}
		\lim_{n\to+\infty}v_n=0\quad\text{in }L^2(\mathbb T).
	\end{equation}
	On the other hand, each $u_n$ satisfies the energy identity \eqref{Eq4.3},
	\[
	\|u_n\|_{\dot H^{r/2}}^2=\varepsilon_n\|u_n\|_{\dot H^{s/2}}^2.
	\]
	Dividing by $\|u_n\|_{\dot H^{s/2}}^2$ yields
	\begin{equation}\label{Eq4.14}
		\|v_n\|_{\dot H^{r/2}}^2=\varepsilon_n\ge \delta.
	\end{equation}
	Now $(v_n)$ is bounded in $H^{s/2}(\mathbb T)$ (indeed $\|v_n\|_{H^{s/2}}^2=\|v_n\|_{L^2}^2+1$), and since $r<s$
	the embedding
	\[
	H^{s/2}(\mathbb T)\hookrightarrow H^{r/2}(\mathbb T)
	\]
	is compact. Hence, up to a subsequence,
	\[
	v_n\to v_\ast\quad\text{strongly in }H^{r/2}(\mathbb T).
	\]
	In particular, $v_n\to v_\ast$ in $L^2(\mathbb T)$, and by \eqref{Eq4.13} we get $v_\ast\equiv 0$.
	Therefore $\|v_n\|_{\dot H^{r/2}}\to 0$, contradicting \eqref{Eq4.14}. This proves \eqref{Eq4.12}.
	
	Since $u\in H_0^s(\mathbb T)$ has mean zero, $\widehat u(0)=0$ and thus $|k|^{s}\ge 1$ for all $k\neq0$ implies
	\[
	\|u\|_{L^2}^2  \leq \|u\|_{\dot H^{s/2}}^2.
	\]
	Hence,
	\[
	\|u\|_{H^{s/2}}^2=\|u\|_{L^2}^2+\|u\|_{\dot H^{s/2}}^2
	\le 2\|u\|_{\dot H^{s/2}}^2
	\le 2C(\delta)^2.
	\]
	Therefore \eqref{Eq4.10} holds with
	\[
	C_1(\delta):=\sqrt{2}\,C(\delta).
	\]
	Since $s>1$, then $s/2>1/2$ and the Sobolev embedding $H^{s/2}(\mathbb T)\hookrightarrow L^\infty(\mathbb T)$ yields
	\[
	\|u\|_{L^\infty}\le C_{\mathrm{emb}}\|u\|_{H^{s/2}}
	\le C_{\mathrm{emb}}\,C_1(\delta).
	\]
	Thus \eqref{Eq4.11} holds with $C_2(\delta):=C_{\mathrm{emb}}\,C_1(\delta)$.
\end{proof}

Finally, we obtain the required $H^s$-estimates.

\begin{lemma}\label{Le4.5}
	Let $\delta\in(0,1)$. Then there exists a constant $C(\delta)>0$ such that every solution
	$(\varepsilon,u)\in(\delta,+\infty)\times H^s_0(\mathbb T)$ of \eqref{Eq3.3} satisfies
	\[
	\|u\|_{H^s}\le C(\delta).
	\]
	In particular, the bound is uniform for $\varepsilon\in[\delta,1)$, whereas for $\varepsilon\ge1$ any solution is trivial.
\end{lemma}

\begin{proof}
	Fix $\delta\in(0,1)$ and let $(\varepsilon,u)\in(\delta,+\infty)\times H^s_0(\mathbb T)$ solve \eqref{Eq3.3}.
	If $u\equiv0$ there is nothing to prove. Otherwise, by Lemma~\ref{Le4.2}(a) we have $\varepsilon\in(0,1)$, hence
	$\varepsilon\in[\delta,1)$. By Lemma~\ref{Le4.4}, there exists $M(\delta)>0$ such that every solution with
	$\varepsilon\in[\delta,1)$ satisfies
	\begin{equation*}
		\|u\|_{H^{s/2}}\le M(\delta).
	\end{equation*}
	Since $s>1$, we have the Sobolev embedding $H^{s/2}(\mathbb T)\hookrightarrow L^\infty(\mathbb T)$, hence
	\begin{equation}\label{Eq4.15}
		\|u\|_{L^\infty}\le C\,M(\delta).
	\end{equation}
	Moreover, since $u\in H^s_0(\mathbb T)$ has mean zero, $\|u\|_{L^2}\le \|u\|_{\dot H^{s/2}}\le M(\delta)$.
	Taking $L^2$--norms in \eqref{Eq3.3}, we obtain that
	\begin{equation}\label{Eq4.16}
		\varepsilon\,\|u\|_{\dot H^{s}}\leq
		\|u\|_{\dot H^{r}}+\|u u_x\|_{L^2}.
	\end{equation}
	Set $X:=\|u\|_{\dot H^s}$.
	If $r\le s/2$, then $\|u\|_{\dot H^{r}}\le \|u\|_{\dot H^{s/2}}\le M(\delta)$.
	If $s/2<r<s$, interpolate between $\dot H^{s/2}$ and $\dot H^s$:
	\[
	\|u\|_{\dot H^{r}}
	\le \|u\|_{\dot H^{s/2}}^{\theta_r}\,\|u\|_{\dot H^{s}}^{1-\theta_r},
	\qquad
	\theta_r=\frac{s-r}{s-s/2}=\frac{2(s-r)}{s},
	\]
	so
	\begin{equation}\label{Eq4.17}
		\|u\|_{\dot H^{r}}
		\le M(\delta)^{\theta_r}\,X^{\beta_r},
		\qquad
		\beta_r:=1-\theta_r=\frac{2r-s}{s}\in(0,1).
	\end{equation}
	Using \eqref{Eq4.15} we have
	\[
	\|u u_x\|_{L^2}\le \|u\|_{L^\infty}\|u_x\|_{L^2}\le C\,M(\delta)\,\|u\|_{\dot H^1}.
	\]
	If $s\ge2$, then $1\le s/2$ and hence $\|u\|_{\dot H^1}\le \|u\|_{\dot H^{s/2}}\le M(\delta)$, so
	\begin{equation}\label{Eq4.18}
		\|u u_x\|_{L^2}\le C\,M(\delta)^2.
	\end{equation}
	If $1<s<2$, then $s/2<1<s$ and we interpolate between $\dot H^{s/2}$ and $\dot H^s$:
	\[
	\|u\|_{\dot H^1}\le \|u\|_{\dot H^{s/2}}^{\theta_1}\,\|u\|_{\dot H^s}^{1-\theta_1},
	\qquad
	\theta_1=\frac{s-1}{s-s/2}=\frac{2(s-1)}{s},
	\]
	hence
	\begin{equation}\label{Eq4.19}
		\|u u_x\|_{L^2}
		\le C\,M(\delta)^{1+\theta_1}\,X^{\beta_1},
		\qquad
		\beta_1:=1-\theta_1=\frac{2-s}{s}\in(0,1).
	\end{equation}
	Combining \eqref{Eq4.16} with \eqref{Eq4.17} and \eqref{Eq4.18}--\eqref{Eq4.19}, and using $\varepsilon\ge\delta$,
	we obtain an inequality of the form
	\begin{equation}\label{Eq4.20}
		\delta X \le A(\delta) + B(\delta) X^{\beta},
	\end{equation}
	where $\beta:=\max\{\beta_r,\beta_1,0\}\in[0,1)$ and $A(\delta),B(\delta)>0$ depend only on $\delta,r,s$
	(via the bound $M(\delta)$), but not on the particular solution.
	
	If $X\le1$ we are done. If $X\ge1$, dividing \eqref{Eq4.20} by $X^\beta$ yields
	\[
	\delta\,X^{1-\beta}\le A(\delta) X^{-\beta}+B(\delta)\le A(\delta)+B(\delta),
	\]
	hence
	\begin{equation*}
		X\le \left(\frac{A(\delta)+B(\delta)}{\delta}\right)^{\frac{1}{1-\beta}}
		=:C_*(\delta).
	\end{equation*}
	Finally,
	\[
	\|u\|_{H^s}\le \|u\|_{L^2}+\|u\|_{\dot H^s}\le M(\delta)+C_*(\delta)=:C(\delta).
	\]
	This completes the proof.
\end{proof}

\section{Global bifurcation theory}\label{Se5}

This section is devoted to study the global structure of the connected component
$$\mathscr{C}_{1}\subset \R_{>0}\times H^{s}_{0}(\mathbb{T}),$$
of nontrivial solutions emanating from the bifurcation point $(\s_{1},0)\in \R_{>0}\times H^{s}_{0}(\mathbb{T})$. See Theorem \ref{Th3.6}. 

We start this section with the required abstract tools to perform a global bifurcation analysis. We state the global alternative theorem for the bifurcation of nonlinear Fredholm operators of index zero given by López-Gómez and Mora-Corral in \cite{LGMC} and sharpened in \cite{LGSM} in the light of the Fitzpatrick, Pejsachowicz and Rabier topological degree \cite{FPRa, FPRb, PR}, a generalization of the Leray--Schauder degree to Fredholm operators of index zero. Throughout this section, we consider a function of class $\mc{C}^{1}$, $\mathfrak{F}:\mathbb{R}_{>0}\times U\to V$, such that
\begin{enumerate}
	\item[(F1)] $\mf{F}$ is orientable in the sense of Fitzpatrick, Pejsachowicz and Rabier, see \cite{FPRa}.
	\item[(F2)] $\mathfrak{F}(\e,0)=0$ for all $\e>0$.
	\item[(F3)] $\partial_{u}\mathfrak{F}(\e,u)\in\Phi_{0}(U,V)$ for every $\e>0$ and $u\in U$.
	\item[(F4)] $\mathfrak{F}$ is proper on bounded and closed subsets of $\mathbb{R}_{>0}\times U$.
	\item[(F5)] $\Sigma(\mf{L})$ is a discrete subset of $\R_{>0}$, where recall that $\mf{L}(\e):=\partial_{u}\mf{F}(\e,0)$, $\e>0$.
\end{enumerate}
The \textit{trivial branch} is the subset 
$$\mc{T}:=\{(\e,0): \e>0\}\subset \R_{>0}\times U,$$
and the \textit{set of non-trivial solutions} is defined by
\begin{equation*}
	\mf{S}=\left[ \mf{F}^{-1}(0)\backslash \mc{T}\right]\cup \{(\e,0):\;\e\in \Sigma(\mf{L})\}.
\end{equation*}
The global alternative theorem reads as follows:
\begin{theorem}[\textbf{Global alternative}]
	\label{Th5.1} Let $\mf{F}\in\mc{C}^{1}(\R_{>0}\times U, V)$ be a map satisfying {\rm(F1)--(F5)}  and $\mathscr{C}$ be a connected component of the set of non-trivial solutions $\mf{S}$ such that $(\e_{0},0)\in\mf{C}$. If the linearization $\mf{L}(\e):=\partial_{u}\mf{F}(\e,0)$, $\e\in \R_{>0}$, is analytic and the oddity condition
	$$\chi[\mf{L},\varepsilon_0]\in 2\N-1,$$
	holds for some $\varepsilon_0\in\Sigma(\mf{L})$, then one of the following non-excluding alternatives occur:
	\begin{enumerate}
		\item[{\rm(1)}] $\mathscr{C}$ is unbounded.
		\item[{\rm(2)}] There exists $\varepsilon_{1}\in \Sigma(\mf{L})$, $\varepsilon_{1}\neq\varepsilon_{0}$, such that $(\varepsilon_{1},0)\in \mathscr{C}$.
		\item[{\rm(3)}] There exists a sequence $\{(\e_{n},u_{n})\}_{n\in\N}\subset \mathscr{C}$ such that $\e_{n}\to 0$ as $n\to+\infty$.
	\end{enumerate}
\end{theorem}

We recall that we have already proved that (F2), (F3) and (F5) hold for our operator \eqref{Eq3.2}. The next result establishes that $\mf{F}$ is proper on closed and bounded subsets of $\R_{>0}\times H^{s}_{0}(\TT)$. In other words, that (F4) holds. This property is important in order to recover some compactness needed for the application of the global bifurcation Theorem \ref{Th5.1}. Recall that a map $f:X\to Y$ between two topological spaces $X, Y$ is \textit{proper} if the preimage of every compact set in $Y$ is compact in $X$.

\begin{lemma}
	The operator $\mf{F}$ is proper on closed and bounded subsets of $\R_{>0}\times H^{s}_{0}(\mathbb{T})$.
\end{lemma}

\begin{proof}
	It suffices to prove that the restriction of $\mf{F}$ to the closed subset
	${K}:=[\e_{-},\e_{+}]\times \bar {B}_{R}$ is proper, where $0<\e_{-}<\e_{+}$ and $B_{R}$ stands for the open ball of $H^{s}_{0}(\mathbb{T})$ of radius $R>0$ centered at $0$. According to
	\cite[Th. 2.7.1]{B}, we must check that $\mf{F}(K)$ is closed in $L^{2}_{0}(\mathbb{T})$, and that, for every $\phi\in L^{2}_{0}(\mathbb{T})$, the set $\mf{F}^{-1}(\phi)\cap {K}$ is compact in $\mathbb{R}_{>0}\times H^{s}_{0}(\mathbb{T})$.
	\par
	To show that $\mf{F}(K)$ is closed in $L^{2}_{0}(\mathbb{T})$, let $\{\phi_{n}\}_{n\in\mathbb{N}}$ be a sequence in $\mf{F}(K)\subset L^{2}_{0}(\mathbb{T})$ such that
	\begin{equation*}
		\lim_{n\to+\infty} \phi_{n}=\phi \quad \text{in } L^{2}_{0}(\mathbb{T}).
	\end{equation*}
	Then, there exists a sequence $\{(\e_{n},u_{n})\}_{n\in\mathbb{N}}$ in $K$ such that \begin{equation}
		\label{Eq5.1}
		\phi_{n}=\mf{F}(\e_{n},u_{n}) \quad \hbox{for all} \;\; n\in\mathbb{N}.
	\end{equation}
	Let $q:=\max\{1,r\}<s$. By the compactness of the embeddings $H^{s}_{0}(\mathbb{T})\hookrightarrow H^{q}_{0}(\mathbb{T})$ and $H^{s}_{0}(\mathbb{T})\hookrightarrow\mc{C}(\mathbb{T})$, we can extract a subsequence $\{(\e_{n_{k}},u_{n_{k}})\}_{k\in\mathbb{N}}$ such that, for some $(\e_{0},u_{0})\in [\e_-,\e_+]\times H^{q}_{0}(\mathbb{T})$,  it holds $\lim_{k\to +\infty} \e_{n_{k}} = \e_{0}$ and
	\begin{equation}
		\label{Eq5.2}
		\lim_{k\to+\infty} u_{n_{k}} =u_{0} \quad \text{in } \mc{C}(\mathbb{T}) \text{ and in }  H^{q}_{0}(\mathbb{T}).
	\end{equation}
	By \eqref{Eq5.1}, we have that $u_{n_k}$ is a weak solution of $\L^{r}u_{n_k}-\e_{n_k}\L^{s}u_{n_k}-u_{n_k}(u_{n_k})_{x}=\phi_{n_k}$. That is, 
	$$
	\int_{\TT}(\L^{r}u_{n_k}-\e_{n_k}\L^{s}u_{n_k})\varphi -\frac{1}{2}\int_{\TT}u^{2}_{n_k}\varphi_{x}=\int_{\mathbb{T}} \phi_{n_k} \varphi, \quad \text{for all} \;\; \varphi\in\mc{C}^{\infty}_{0}(\mathbb{T}). 
	$$
	Then, by \eqref{Eq5.2}, taking $k\to+\infty$ we obtain
	$$
	\int_{\TT}(\L^{r}u_{0}-\e_{0}\L^{s}u_{0})\varphi -\frac{1}{2}\int_{\TT}u^{2}_{0}\varphi_{x}=\int_{\mathbb{T}} \phi \varphi, \quad \text{for all} \;\; \varphi\in\mc{C}^{\infty}_{0}(\mathbb{T}). 
	$$
	Therefore $u_{0}$ must be a weak solution of 
	\begin{equation}
		\label{Eq5.3}
		\L^{r}u_{0}-\e_{0}\L^{s}u_{0}-u_{0}(u_{0})_{x}=\phi.
	\end{equation}
	By the regularity result stated in Lemma \ref{Le4.1}, we deduce that $u_{0}\in H^{s}_{0}(\mathbb{T})$ and $\phi=\mf{F}(\e_{0},u_{0})$.	 Therefore,
	\begin{align*}
		&\L^{r}(u_{n_{k}}-u_{0})-(\e_{n_{k}}-\e_{0})\L^{s}u_{n_{k}}-\e_{0}\L^{s}(u_{n_k}-u_0)\\ 
		&-u_{n_k}(u_{n_k}-u_{0})_{x} - (u_{0})_{x}(u_{n_k}-u_{0})=\phi_{n_{k}}-\phi, \quad x\in \mathbb{T}.
	\end{align*}
Furthermore, by \eqref{Eq4.1} and the boundedness of $K$, we obtain the bound
\begin{align*}
\varepsilon_0 \|u_{n_{k}}-u_{0}\|_{\dot{H}^{s}}\leq \; & |\e_{n_{k}}-\e_{0}|\|u_{n_{k}}\|_{\dot{H}^{s}}+\|u_{n_{k}}-u_{0}\|_{L^\infty}\|u_0\|_{\dot{H}^1}\\
&+\|u_{n_{k}}-u_{0}\|_{\dot{H}^1}\|u_0\|_{L^\infty}+\|u_{n_{k}}-u_{0}\|_{\dot{H}^r}+\|\phi_{n_{k}}-\phi\|_{L^{2}}\\
\leq& \; C(|\e_{n_{k}}-\e_{0}|+\|u_{n_{k}}-u_{0}\|_{\dot{H}^q}+\|\phi_{n_{k}}-\phi\|_{L^{2}}),
\end{align*}
for some positive constant $C>0$.
	Therefore, as $K$ is closed, we infer $(\e_{0},u_{0})\in K$. This proves that $\phi\in \mf{F}(K)$.
	\par
	Now, pick $\phi\in L^{2}_{0}(\mathbb{T})$. To show that $\mf{F}^{-1}(\phi)\cap K$ is compact in $[\e_-,\e_+]\times H^{s}_{0}(\mathbb{T})$, let $\{(\e_{n},u_{n})\}_{n\in\mathbb{N}}$ be a sequence in $\mf{F}^{-1}(\phi)\cap K$. Then,
	\begin{equation*}
		\mf{F}(\e_{n},u_{n})=\phi \quad \hbox{for all}\;\; n\in\mathbb{N}.
	\end{equation*}
	Based again on the compactness of the imbedding $H^{s}_{0}(\mathbb{T}) \hookrightarrow \mc{C}(\mathbb{T})$, we can extract a subsequence $\{(\e_{n_{k}},u_{n_{k}})\}_{k\in\mathbb{N}}$ such that, for some $(\e_{0},u_{0})\in [\e_-,\e_+]\times \mc{C}(\mathbb{T})$,   $\lim_{k\to +\infty} \e_{n_{k}} = \e_{0}$ and \eqref{Eq5.2} holds.
	Similarly,  $u_{0}\in \mc{C}(\mathbb{T})$ is a weak solution of \eqref{Eq5.3} and, by regularity,
	$u_{0}\in H^{s}_{0}(\mathbb{T})$ and $\mf{F}(\e_{0},u_{0})=\phi$. In particular, for every $k\in\mathbb{N}$,
	\begin{align*}
		&\L^{r}(u_{n_{k}}-u_{0})-(\e_{n_{k}}-\e_{0})\L^{s}u_{n_{k}}-\e_{0}\L^{s}(u_{n_k}-u_0)\\ 
		&-u_{n_k}(u_{n_k}-u_{0})_{x} - (u_{0})_{x}(u_{n_k}-u_{0})=0, \quad x\in \mathbb{T}.
	\end{align*}
	By the regularity result of Proposition \ref{Le4.1}, we have the estimate
	\begin{equation*}
		\|u_{n_{k}}-u_{0}\|_{\dot{H}^{s}}\leq C(|\e_{n_{k}}-\e_{0}|+\|u_{n_{k}}-u_{0}\|_{\dot{H}^q}), \quad k\in\N,
	\end{equation*}
	for some positive constant $C>0$.  Therefore, letting $k\to +\infty$ we finally get
	that
	$$
	\lim_{k\to +\infty} (\e_{n_{k}},u_{n_{k}}) = (\e_{0},u_{0})\quad \hbox{in}\;\;
	[\e_-,\e_+] \times H^{s}_{0}(\mathbb{T}).
	$$
	This concludes the proof.
\end{proof}

The next result shows that the connected components $\mathscr{C}_{k}$, $k\geq 1$, live in $(0,1]\times H^{s}_{0}(\mathbb{T})$.
\begin{lemma}
	\label{Le5.3}
	For each $k\in\N$, the following set inclusion holds:
	\begin{equation}
		\label{Eq5.4}
		\mathscr{C}_{k}\subset (0,1]\times H^{s}_{0}(\mathbb{T}).
	\end{equation}
\end{lemma}
\begin{proof}
Fix $k\in\N$. Recall that $\mathscr{C}_{k}\subset \mf{S}$ and 
$$
\mf{S}\cap \mc{T} = \left\{(\s_{k},0) : k\in\N\right\}\subset (0,1]\times \{0\}.
$$	
Hence, by Lemma \ref{Le4.2}, item (a), necessarily $\mf{S}\subset (0,1]\times H^{s}_{0}(\TT)$ and consequently the inclusion \eqref{Eq5.4} holds.
\end{proof}

Let us state and prove the main result of this paper:
\begin{theorem}
	\label{Th5.4}
	It holds that 
	\begin{equation}
		\label{Eq5.5}
		(2^{r-s},1]\subset \mc{P}_{\e}(\mathscr{C}_{1}),
	\end{equation}
	where $\mc{P}_{\e}:\R_{>0}\times H^{s}_{0}(\mathbb{T})\to \R$, $(\e,u)\mapsto \e$, is the $\e$-projection operator.
\end{theorem}

\begin{proof}
	We apply Theorem \ref{Th5.1} to the nonlinearity 
	\begin{equation*}
		\mf{F}:\R_{>0}\times H_{0}^{s}(\mathbb{T})\longrightarrow L_{0}^{2}(\mathbb{T}), \quad \mf{F}(\e,u)=\L^{r}u-\e \L^{s}u - uu_{x}.
	\end{equation*}
As we have already stated, (F2)--(F5) hold for our operator. We have to verify (F1). Indeed, $\mf{F}$ is orientable in the sense of Fitzpatrick, Pejsachowicz and Rabier since the domain $\R_{>0}\times H^{s}_{0}(\mathbb{T})$ is simply connected (see for instance \cite[Pr. 1.9]{FPRa}). Then, the application of Theorem \ref{Th5.1} to the connected component $\mathscr{C}_{1}$ implies that or $\mathscr{C}_{1}$ is unbounded or there exists $\s_{m}\in \Sigma(\mf{L})$, $m\neq 1$, such that $(\s_{m},0)\in\mathscr{C}_{1}$ or there exists a sequence $\{(\e_{n},u_n)\}_{n\in\N}\subset \mathscr{C}_{1}$ such that $\e_n\to 0$.

If the second statement holds then \eqref{Eq5.5} holds trivially. Indeed, in this case, $(\s_1,0), (\s_m,0)\in \mathscr{C}_{1}$ for some $m>1$. Since $\mathscr{C}_{1}$ is connected and $\mc{P}_{\e}$ is continuous, we infer that $\mc{P}_{\e}(\mathscr{C}_{1})$ is a connected subset of $\R$. Therefore, $\mc{P}_{\e}(\mathscr{C}_{1})=I_{\a,\b}$, where $I_{\a,\b}$ is an interval of $\R$ with boundary $\{\a,\b\}$ for some $\alpha\leq \beta$. Moreover, $\alpha \leq \s_{2} < \s_{1}\leq \beta$. Therefore, $(\s_{2},\s_{1}]\subset I_{\a,\b}= \mc{P}_{\e}(\mathscr{C}_{1})$ and \eqref{Eq5.5} is proven.

Suppose that the third alternative of Theorem \ref{Th5.1} holds. Then, there exists a sequence $\{(\e_{n},u_n)\}_{n\in\N}\subset \mathscr{C}_{1}$ with $\e_n\to 0$. Choose $n_0\in\N$ such that $\e_{n_0} < \s_2$. Again, since $\mc{P}_{\e}(\mathscr{C}_{1})$ is connected, necessarily $\mc{P}_{\e}(\mathscr{C}_{1})=I_{\a,\b}$. Moreover, since $(\e_{n_0},u_{n_0})\in \mathscr{C}_{1}$ and $\e_{n_0}<\s_{2}$, we have $\alpha \leq \e_{n_0} < \s_{2} < \s_{1}\leq \beta$. Therefore, $(\s_{2},\s_{1}]\subset I_{\a,\b}= \mc{P}_{\e}(\mathscr{C}_{1})$ and \eqref{Eq5.5} is proven.

Finally, suppose the first alternative of Theorem \ref{Th5.1} holds, that is, suppose that $\mathscr{C}_{1}$ is unbounded. In order to reach a contradiction, suppose also that \eqref{Eq5.5} does not hold. Then, there exist $\a,\b$ with $\s_{2}<\a<\b \leq \s_{1}$ such that 
	$$
	\mc{P}_{\e}(\mathscr{C}_{1}) = I_{\a,\b},
	$$
	where $I_{\a,\b}$ is an interval of $\R$ with boundary $\{\a,\b\}$. The existence and the estimate of $\b$ is justified by Lemma \ref{Le5.3}. In any case, by the unboundedness of $\mathscr{C}_{1}$, there exists a sequence $\{(\e_n,u_n)\}_{n\in\N}\subset \mathscr{C}_{1}$ such that $\e_n\to\e_{0}\in \overline{I_{\a,\b}}$ as $n\to+\infty$ and
	$$
	\lim_{n\to+\infty}\|u_n\|_{H^{s}}=+\infty.
	$$
	But this contradicts Lemma \ref{Le4.5} and concludes the proof.
\end{proof}

We conclude the global analysis by establishing an alternative of the behavior of the solutions when $\e\to 0^+$. 

\begin{theorem}
	\label{Th5.5}
	Let $(\varepsilon_n,u_n)\in\R_{>0}\times H^s_0(\TT)$ be a sequence of nontrivial solutions of \eqref{Eq3.3} such that
	\[
	\lim_{n\to+\infty}\varepsilon_n = 0.
	\]
	Then the following two (non–excluding) alternatives hold:
	\begin{enumerate}
		\item[\textup{(i)}] $\displaystyle \limsup_{n\to+\infty}\|u_n\|_{\dot H^{s/2}}=+\infty$;
		\item[\textup{(ii)}] $\displaystyle \lim_{n\to+\infty}\|u_n\|_{L^\infty}=0$.
	\end{enumerate}
\end{theorem}

\begin{proof}
	Assume that (i) fails, i.e.\ there exists $M>0$ such that
	\begin{equation}\label{Eq5.6}
		\|u_n\|_{\dot H^{s/2}}\le M \qquad\text{for all }n.
	\end{equation}
	Let us show that $\|u_n\|_{L^{2}}\to 0$ as $n\to +\infty$. Using \eqref{Eq4.2} and $\varepsilon_n\to0$, we deduce
	\begin{equation*}
		\lim_{n\to +\infty}\frac{\|u_{n}\|_{\dot{H}^{r/2}}}{\|u_{n}\|_{\dot{H}^{s/2}}}=\lim_{n\to +\infty}\sqrt{\varepsilon_{n}}=0.
	\end{equation*}
	Therefore,
	\begin{equation}
		\label{Eq5.7}
		\lim_{n\to +\infty}\|u_{n}\|_{\dot{H}^{r/2}}=0.
	\end{equation}
	Suppose first that $r>0$. Since $u_n$ has zero mean,
	\[
	\|u_n\|_{L^2}\leq \|u_n\|_{\dot H^{r/2}}.
	\]
	Together with \eqref{Eq5.7}, this gives $\|u_n\|_{L^2}\to0$.
	
	If $r\le0$, we argue by contradiction. Suppose $\|u_n\|_{L^2}\not\to0$. Then there exist $c>0$ and a subsequence
	(still denoted $u_n$) such that
	\begin{equation}\label{Eq5.8}
		\|u_n\|_{L^2}\ge c \qquad\text{for all }n.
	\end{equation}
	Moreover, since $u_n\in H^s_0(\TT)$ has zero mean, again
	\[
	\|u_n\|_{L^2} \leq \|u_n\|_{\dot H^{s/2}}.
	\]
	Hence \eqref{Eq5.6} implies $\|u_n\|_{L^2}\le M$.
	By this bound and the compactness of the embedding
	$H^{s/2}(\TT)\hookrightarrow L^2(\TT)$, we may extract a further subsequence such that
	\[
	\lim_{n\to+\infty}u_n =  u_\ast \quad\text{ in }L^2(\TT).
	\]
	In particular, $\|u_\ast\|_{L^2}\ge c$ by \eqref{Eq5.8}. On the other hand, when $r\le0$ we have $|k|^{r}\le1$
	for $k\neq0$, and therefore strong $L^2$ convergence implies strong $\dot H^{r/2}$ convergence.
	Combining with \eqref{Eq5.7} yields $\|u_\ast\|_{\dot H^{r/2}}=0$, hence
	\[
	0=\|u_\ast\|_{\dot H^{r/2}}^2
	=\sum_{k\neq0}|k|^{r}|\widehat u_\ast(k)|^2.
	\]
	Since $|k|^{r}>0$ for every $k\neq0$, we obtain $\widehat u_\ast(k)=0$ for all $k\neq0$, and since $u_\ast$
	has zero mean, also $\widehat u_\ast(0)=0$. Thus $u_\ast\equiv0$, contradicting $\|u_\ast\|_{L^2}\ge c$.
	Therefore $\|u_n\|_{L^2}\to0$.
	
	Finally, since $s>1$, from the compactness of the embedding $H^{s/2}(\TT)\hookrightarrow \mc{C}(\TT)$, we may extract, from any subsequence of $(u_n)$, a further subsequence converging
	in $\mc{C}(\TT)$ to some limit $v$. But we already proved $u_n\to0$ in $L^2(\TT)$, hence necessarily $v\equiv0$.
	This concludes the proof.
\end{proof}

\section{Numeric Analysis}\label{Se6}
\label{B}

%# títulos que poner a los diagramas del caso no eliptico
%# S0: Series of solution profiles along ... for \e \in [0.6178,1]
%# S1: Series of solution profiles  along ...  for \e \in [0.1629,.25]
%# S2: Series of solution profiles  along ... for \e \in [0.0779,0.111]
%# S3: Series of solution profiles  along ... for \e \in [0.0446,0.0625]
%# S4: Series of solution profiles  along ... for \e \in [0.030,0.040]

%# títulos que poner a los diagramas del caso eliptico
%# S0: Series of solution profiles  along ... for \e \in [0.1581,1]
%# S1: Series of solution profiles  along ... for \e \in [0.1508,.50]
%# S2: Series of solution profiles  along ... for \e \in [0.1477,0.3333]
%# S3: Series of solution profiles  along ... for \e \in [0.1492,0.2500]
%# S4: Series of solution profiles  along ... for \e \in [0.1499,0.2000]
As usual in numerical simulations, a previous discretization of the problem is necessary in order to reduce the original continuous problem to a finite dimensional one. In most of the previous works devoted to numerical continuation in bifurcation theory,  a pseudo-spectral method combining a spectral method with collocation has usually been used to discretize the problem, especially to compute small solutions bifurcating from $u=0$ (see, e.g., G\'{o}mez-Re\~{n}asco and L\'{o}pez-G\'{o}mez  \cite{GRLGNA},  L\'{o}pez-G\'{o}mez, Molina-Meyer and Tellini \cite{LGMMT2} and L\'opez-G\'omez, Molina-Meyer and Rabinowitz \cite{LGMMR}). This method gives high accuracy at a reasonable computational cost (see, e.g., Canuto, Hussaini, Quarteroni and Zang \cite{CHQZ}). In particular, it is very efficient for choosing the shooting direction from the trivial solution in order to compute the bifurcated small classical solutions, as it provides us with the true values of the first bifurcation points (see Eilbeck \cite{Ei}).

\par
However, we have preferred in this article to use centered finite difference schemes, as described below, in order to discretize the problem, since bifurcation values from the trivial branch are well known (see Theorem \ref{Th3.6}) and easy to compute for the particular case for which simulations have been performed. The domain of the problem is first discretized using a regular mesh of equidistant points $(x_0, \ldots, x_{N+1}) \in [0,\pi]^{N+2}$ with $N\geq 1$. If we denote by $h$ the step size of the mesh, we get:
\begin{equation*}
	h=\frac{\pi}{N+1} \qquad \text{and} \qquad x_i = ih, \qquad i \in \{0, \ldots ,N+1\}.
\end{equation*}
Let us also denote by $u_i \approx u(x_i)$, the approximate values of the unknown solution $u$ at each node of the mesh.

\par
We then discretize the left and right hand side operators of equation \eqref{Eq1.2}, taking into account the fact that the solutions we are seeking must be odd and periodic. Since the right hand side contains a Fourier multiplier, a numerical quadrature using the values $u_i$ is necessary. In this paper, we chose a simple trapezoidal rule, of order 2, for such quadrature. For the left hand side, involving the derivative of the unknown $u$ with respect to the spatial variable $x$, a centered finite difference scheme, also of order 2, was used.

\par
The discrete problem consists then in finding all possible solutions $(\e,u_h(\e)) \in \mathbb{R} \times \mathbb{R}^N$ of the discretized equation $\mf{F}_N(\e,u_h)=0$ where $N$ is a parameter that depends on the chosen mesh. These solutions will converge to the solutions $(\e,u)$ of the original problem when $N \to +\infty$  or $h \to 0$. Such a convergence was proved, for general Galerkin approximations, by Brezzi, Rappaz and Raviart \cite{BRR1,BRR2,BRR3}, and by L\'{o}pez-G\'{o}mez,  Molina-Meyer and  Villareal \cite{LGMMV}. These results are based on the implicit function theorem. As a consequence, in these situations, the local structure of the solution sets for the continuous and the discrete models are equivalent provided that a sufficiently fine discretization is performed. 
In this paper, we have chosen $N=100$ and deemed it sufficient, taking into account the computation complexity of the problem. Indeed, the right hand side operator, once discretized, is a fully dense matrix of size $100$ by $100$ as opposed to simpler operators such as the Laplacian, whose corresponding matrix, once discretized, is a sparse (tridiagonal) matrix.

\par
In order to compute the components of solutions for the particular case $r=\tfrac{1}{2}$ and $s=\tfrac{3}{2}$, we started from a value of $\e_0^k$ equal to $k^{-1}$, together with an initial guess $u_0^k(x)=\sin(k\pi x)$. The choice of such an initial iterate is motivated by Theorem \ref{Th3.6}. Applying then the correction algorithm, based on Newton's method, we obtain the approximate solution of the problem for the initial value of $\e_0^k$. Finally, by performing a global continuation based on the algorithm of Keller and Yang \cite{KY} and detecting whether some secondary bifurcation occurs, we were able to continue the computation of the component of solutions. As a stopping criterion for Newton's method, we considered an iteration to be satisfactory when the infinity norm of our \emph{augmented system}, evaluated at such an iteration, is smaller than $0.0001$. The \lq\lq augmented system" refers to the system associated with the algorithm of Keller and Yang \cite{KY}, which makes regular any (singular) turning point through an appropriate parametrization by a pseudo-length of arc of curve. 

\par
Computing the components of solutions for small values of $\e$ is extremely challenging since these solutions develop narrow boundary layers where the gradients become extremely large. This phenomenon brings about numerical instabilities in form of high frequency oscillations inside these layers. The profiles of solutions depicted in this paper are those corresponding to values of $\e$ for which the computations are free of instabilities. In order to plot the bifurcation diagrams, we decided to depict in ordinates the discrete $L^2$ norm versus $\e$ in abscissa. Indeed, when these high frequency oscillations occur, the smoothing $L^2$ norm behaves \emph{well} and allows to plot these diagrams for a wider range of $\e$, despite these high frequency oscillations. In addition, for some of the main bifurcated branches we have computed, i.e those emanating from the trivial branch $u=0$, secondary bifurcations have been detected but their analysis has been left out of this paper and has been postponed to a future work.

\par
In view of the numerical results presented in this paper (see also Theorem \ref{Th5.5}), it is reasonable to conjecture that solutions blow up in $H^s$ to $+\infty$ as $\e \to 0^{+}$. Indeed, as $\e$ decreases, the values for which the extremum of solutions along components $\mathscr{C}_k$ is reached, tend to concentrate in the vicinity of some points (see Figure \ref{F2}). For instance, for the profiles of the connected component $\mathscr{C}_{1}$, the gradient blows up in the vicinity of $x=-\pi,\pi$. This phenomenon originates boundary layers, as mentioned previously, where the gradient of solutions become very large. In addition, all solutions along the same component share the same zeros, independently of the values of $\e$.

\appendix

\section{Preliminaries on nonlinear spectral theory}
\label{A1}
In this subsection we collect some fundamental concepts about nonlinear spectral theory that will be used throughout the article. We start with some definitions. A \emph{Fredholm operator family} is any continuous map $\mathfrak{L}\in \mathcal{C}([a,b],\Phi_{0}(U,V))$.  For any given $\mathfrak{L}\in \mathcal{C}([a,b],\Phi_{0}(U,V))$, it is said that $\e\in[a,b]$ is a \emph{generalized eigenvalue} of $\mathfrak{L}$ if $\mathfrak{L}(\e)\notin GL(U,V)$, and the \emph{generalized spectrum} of $\mathfrak{L}$, $\Sigma(\mathfrak{L})$,  is defined through   	
\begin{equation*}
	\Sigma(\mathfrak{L}):=\{\e\in [a,b]: \mathfrak{L}(\e)\notin GL(U,V)\}.
\end{equation*}
One of the cornerstones of spectral theory is the concept of algebraic multiplicity. Algebraic multiplicity is classically defined for eigenvalues of compact operators in Banach spaces. Let $K: U\to U$ be a compact linear operator on the $\mathbb{K}$-Banach space $U$, $\mathbb{K}\in\{\mathbb{R},\mathbb{C}\}$. We set
\begin{equation*}
	\mathfrak{L}(\e):=K-\e I_U, \quad \e\in J_{\e_0},
\end{equation*}
where $J_{\e_0}\subset \mathbb{K}$ is a neighbourhood of $\e_0$ and $\e_0$ is an eigenvalue of $K$. Then, the
classical algebraic multiplicity\index{multiplicity!classical} of
$\e_0$ as an eigenvalue of $K$ is defined through
\begin{equation*}
	\mf{m}_{\alg}[K,\e_0] := \dim \bigcup_{\mu=1}^\infty
	N[(K-\e_0  I_U)^\mu]
	\label{(8.2)}.
\end{equation*}
By Fredholm's theorem, for $\e_{0}\neq 0$, $\mf{L}:J_{\e_0}\to \Phi_{0}(U)$ where $\Phi_{0}(U)$ denotes the space of Fredholm operators of index zero $T:U\to U$. 
\par In 1988, J. Esquinas and J. L\'{o}pez-G\'{o}mez \cite{Es,ELG}, inspired by the work of Krasnoselski \cite{Kr}, Rabinowitz \cite{Ra} and Magnus \cite{Ma}, generalized the concept of algebraic multiplicity to every Fredholm operator family $\mf{L}:[a,b]\to \Phi_{0}(U,V)$, not necessarily of the form $\mathfrak{L}(\e)=K-\e I_U$, $U=V$, with $K$ compact. They denote it by $\chi[\mf{L},\e_{0}]$ and they proved that this concept is consistent with the classical algebraic multiplicity, that is, $\chi[K-\e I_{U},\e_{0}]=\mf{m}_{\alg}[K,\e_0]$ with $K$ compact and that it shares many properties of its classical counterpart. This theory was subsequently refined in the monograph \cite{LG01}.
We proceed to define the generalized algebraic multiplicity through the concept of algebraic eigenvalue going back to \cite{LG01}. 

\begin{definition}
	Let $\mathfrak{L}\in \mathcal{C}([a,b], \Phi_{0}(U,V))$ and $\kappa\in\mathbb{N}$. A generalized eigenvalue $\e_{0}\in\Sigma(\mathfrak{L})$ is said to be $\kappa$-algebraic if there exists $\varepsilon>0$ such that
	\begin{enumerate}
		\item[{\rm (a)}] $\mathfrak{L}(\e)\in GL(U,V)$ if $0<|\e-\e_0|<\varepsilon$;
		\item[{\rm (b)}] There exists $C>0$ such that
		\begin{equation}
			\label{2.2}
			\|\mathfrak{L}^{-1}(\e)\|_{\mc{L}}<\frac{C}{|\e-\e_{0}|^{\kappa}}\quad\hbox{if}\;\;
			0<|\e-\e_0|<\varepsilon;
		\end{equation}
		\item[{\rm (c)}] $\kappa$ is the minimal integer for which \eqref{2.2} holds.
	\end{enumerate}
\end{definition}
Throughout this paper, the set of $\kappa$-algebraic eigenvalues of $\mathfrak{L}$ is  denoted by $\Alg_\kappa(\mathfrak{L})$, and the set of \emph{algebraic eigenvalues} by
\[
\Alg(\mathfrak{L}):=\bigcup_{\kappa\in\mathbb{N}}\Alg_\kappa(\mathfrak{L}).
\]
We will construct an infinite dimensional analogue of the classical algebraic multiplicity for algebraic eigenvalues. It can be carried out through the theory of Esquinas and L\'{o}pez-G\'{o}mez
\cite{ELG},  where the following pivotal concept, generalizing the transversality condition of
Crandall and Rabinowitz \cite{CR},  was introduced. Throughout this paper, we set
$$\mathfrak{L}_{j}:=\frac{1}{j!}\mathfrak{L}^{(j)}(\e_{0}), \quad 1\leq j\leq r,$$
should these derivatives exist.

\begin{definition}
	Let $\mathfrak{L}\in \mathcal{C}^{r}([a,b],\Phi_{0}(U,V))$ and $1\leq \kappa \leq r$. Then, a given $\e_{0}\in \Sigma(\mathfrak{L})$ is said to be a $\kappa$-transversal eigenvalue of $\mathfrak{L}$ if
	\begin{equation*}
		\bigoplus_{j=1}^{\kappa}\mathfrak{L}_{j}\left(\bigcap_{i=0}^{j-1}N(\mathfrak{L}_{i})\right)
		\oplus R(\mathfrak{L}_{0})=V\;\; \hbox{with}\;\; \mathfrak{L}_{\kappa}\left(\bigcap_{i=0}^{\kappa-1}N(\mathfrak{L}_{i})\right)\neq \{0\}.
	\end{equation*}
\end{definition}

For these eigenvalues, the following generalized concept of algebraic multiplicity can be introduced:
\begin{equation*}
	\chi[\mathfrak{L}, \e_{0}] :=\sum_{j=1}^{\kappa}j\cdot \dim \mathfrak{L}_{j}\left(\bigcap_{i=0}^{j-1}N[\mathfrak{L}_{i}]\right).
\end{equation*}
In particular, when $N[\mf{L}_0]=\mathrm{span}\{\v\}$ for some $\v\in U$ such that $\mf{L}_1\v\notin R[\mf{L}_0]$, then
\begin{equation}
	\label{ii.4}
	\mf{L}_1(N[\mf{L}_0])\oplus R[\mf{L}_0]=V
\end{equation}
and hence, $\l_0$ is a 1-transversal eigenvalue of $\mf{L}(\l)$ with $\chi[\mf{L},\l_0]=1$. The transversality condition \eqref{ii.4} goes back to Crandall and Rabinowitz \cite{CR}. More generally, under condition \eqref{ii.4}, it holds that
\begin{equation*}
	\chi[\mf{L},\l_0]=\dim N[\mf{L}_0].
\end{equation*}
According to Theorems 4.3.2 and 5.3.3 of \cite{LG01}, for every $\mathfrak{L}\in \mathcal{C}^{r}([a,b], \Phi_{0}(U,V))$, $\kappa\in\{1,2,...,r\}$ and $\e_{0}\in \Alg_{\kappa}(\mathfrak{L})$, there exists a polynomial $\Phi: [a,b]\to \mathcal{L}(U)$ with $\Phi(\e_{0})=I_{U}$ such that $\e_{0}$ is a $\kappa$-transversal eigenvalue of the path $\mathfrak{L}^{\Phi}(\e):=\mathfrak{L}(\e)\circ\Phi(\e)$
and $\chi[\mathfrak{L}^{\Phi},\e_{0}]$ is independent of the curve of \emph{trasversalizing local isomorphisms} $\Phi$ chosen to transversalize $\mathfrak{L}$ at $\e_0$. Therefore, the following concept of multiplicity
is consistent
\begin{equation*}
	\chi[\mf{L},\e_0]:= \chi[\mathfrak{L}^{\Phi},\e_{0}],
\end{equation*}
and it can be easily extended by setting
$\chi[\mathfrak{L},\e_0] =0$ if $\e_0\notin\Sigma(\mathfrak{L})$ and
$\chi[\mathfrak{L},\e_0] =+\infty$ if $\e_0\in \Sigma(\mathfrak{L})
\setminus \Alg(\mathfrak{L})$ and $r=+\infty$. Thus, $\chi[\mathfrak{L},\e]$ is well defined for all  $\e\in [a,b]$ of any smooth path $\mathfrak{L}\in \mathcal{C}^{\infty}([a,b],\Phi_{0}(U,V))$.

\end{document}